\documentclass[11pt]{amsart}
\usepackage{amsmath,amssymb,amsthm,amsfonts,amstext,amsbsy,amscd}
\usepackage[utf8]{inputenc}
\usepackage{bbm,dsfont}
\usepackage{color}
\usepackage{comment}

\usepackage[colorlinks,citecolor=blue,urlcolor=blue,linkcolor = blue,hypertexnames=false]{hyperref}

 \newtheorem{thm}{Theorem}[section]
 \newtheorem{cor}[thm]{Corollary}
 \newtheorem{lem}[thm]{Lemma}
 \newtheorem{prop}[thm]{Proposition}

 \theoremstyle{definition}
 
 \newtheorem{rem}[thm]{Remark}
 
 \newtheorem{ass}[thm]{Assumption}

\numberwithin{equation}{section}

 \newcommand{\ind}{\mathbbm{1}}
 \newcommand{\1}{\mathbbm{1}}

 \newcommand{\R}{\mathbb{R}}
 \newcommand{\Z}{\mathbb{Z}}

 \newcommand{\ic}{\mathrm{i}}

 \newcommand{\E}{\mathbb{E}}

 \newcommand{\N}{\mathbb{N}}
 
 \newcommand{\F}{\mathcal{F}}

 \newcommand{\ad}{\mathfrak{a}}
 \newcommand{\bd}{\mathfrak{b}}

\renewcommand{\P}{\mathbb{P}}

\newcommand{\Atwo}{{\bf (A1)}}
\newcommand{\Athree}{{\bf (A2)}}
\newcommand{\Afour}{{\bf (A3)}}

\newcommand{\co}{\mathrm{c}}

\newcommand{\MJ}{}

\title{Weak dependence and optimal quantitative self-normalized central limit theorems}
\author{Jirak,\phantom{x}Moritz}

\begin{document}

\begin{abstract}
Consider a stationary, weakly dependent sequence of random variables. Given only mild conditions, allowing for polynomial decay of the autocovariance function, we show a Berry-Esseen bound of optimal order $n^{-1/2}$ for studentized (self-normalized) partial sums, both for the Kolmogorov and Wasserstein (and $L^p$) distance. The results show that, in general, (minimax) optimal estimators of the long-run variance lead to suboptimal bounds in the central limit theorem, that is, the rate $n^{-1/2}$ cannot be reached. This can be salvaged by simple methods: In order to maintain the optimal speed of convergence $n^{-1/2}$, simple over-smoothing within a certain range is necessary and sufficient. The setup contains many prominent dynamical systems and time series models, including random walks on the general linear group, products of positive random matrices, functionals of Garch models of any order, functionals of dynamical systems arising from SDEs, iterated random functions and many more.
\end{abstract}

\maketitle

\section{Introduction}

Let $(X_k)_{k \in \Z}$ be a stationary, weakly dependent sequence with zero mean. For decades, an important question has been whether the central limit theorem applies to the (normalized) partial sum, and if so, how fast convergence takes place (in the uniform metric), that is, for some $r_n \to 0$, we have
\begin{align}\label{intro:eq:1}
\sup_{x \in \R}\Big|\P\Big(\frac{X_1 + \ldots + X_n}{(n \sigma^2)^{1/2}} \leq x  \Big) - \Phi\big(x\big)\Big| \leq r_n, \quad \sigma^2 > 0,
\end{align}
where $\sigma^2$ denotes the long-run variance, see for instance \cite{bolthausen_1982_ptrf}, \cite{GOUEZEL2005:hip}, \cite{herve_pene_2010_bulletin}, \cite{jirak_be_aop_2016}, \cite{Kloeckner:2019:aap} or \cite{Rio_1996} for some classic and more recent results in this area. However, from a more practical point of view, the actual quantity of interest is the \textit{studentized} (self-normalized) version
\begin{align}\label{intro:eq:2}
\frac{X_1 + \ldots + X_n}{(n \hat{\sigma}_{nb}^2)^{1/2}},
\end{align}
since the long-run variance $\sigma^2$ is typically unknown. Here, $\hat{\sigma}_{nb}^2$ is an appropriate, consistent estimator, typically involving a bandwidth $b$. While consistent estimation is sufficient for the validity of a central limit theorem, studentization (self-normalization) has quite interesting effects on concentration properties of \eqref{intro:eq:2} and its rate of convergence. In case of i.i.d. random variables, this is by now quite well understood. A thorough study of the matter is considerably more involved though than the classical case, since studentization adds a rather unpleasant nonlinear aspect to the problem. On the other hand, studentization also offers positive effects, allowing, for instance, for the validity of Cram\'{e}r-type moderate deviation principles subject only to very few moments, see for instance \cite{goetze:chistyakov:aop:2004}, \cite{shao_self_book_2009}, \cite{fan:qui:man:shao:bernoulli:2019}, \cite{hall:aop:1988}, \cite{shao:jing:aop:2003}, \cite{shao:survey:2013} for some accounts, results and extensions.
However, as was pointed out by a reviewer, even for moderate sizes of $n$ the quality of such bounds may be insufficient for reliable statistical inference, see \cite{Pinelis:2017:student:BE} for a detailed discussion.

In stark contrast to the i.i.d. setup, there are almost no (general) results in the literature regarding the weakly dependent case in this context. In particular, despite highly influential works such as \cite{andrews} or \cite{newey:west:1994} (which we will discuss below), how to select $\hat{\sigma}_{nb}^2$ (and $b$) is largely unresolved, even in the case of linear processes. In ~\cite{bentkus_1997}, suboptimal rates are obtained for exponentially decaying $\beta$-mixing sequences. Studentized Edgeworth-type expansions are developed in \cite{goetze_kunsch_1996}, \cite{lahiri:2007:aos}, \cite{lahirir:2010:aos}, but the necessary assumptions are quite restrictive and hard to verify, including a conditional Cram\'{e}r-type condition, among others. Moreover, again an exponential decay of the mixing coefficients, used to describe the weak dependence, is required. More recently, ~\cite{chen2016:aos} (see also \cite{qui:shao:fan:grama:spa:20205124}, \cite{shao:aos:2022}) obtained deviation bounds of Cram\'{e}r-type by various blocking schemes, where the best bounds utilize ``throwing away blocks``, that is, introducing artificial gaps in the dependence structure. Unfortunately though, the obtained speed of convergence is not optimal, and again an exponential decay of weak dependence is used. On the other hand, ~\cite{chen2016:aos} also obtain an interesting negative result, namely that deviation bounds of Cram\'{e}r-type are generally \textit{not possible} in the presence of polynomially decaying weak dependence. Sadly, this phenomenon already is manifested by linear processes, so there appears to be little hope for general Cram\'{e}r-type deviation bounds in this case.

The aim of this note is thus to establish the optimal Berry-Esseen bound
\begin{align}\label{intro:eq:3}
\sup_{x \in \R}\Big|\P\Big(\frac{X_1 + \ldots + X_n}{(n \hat{\sigma}_{nb}^2)^{1/2}} \leq x  \Big) - \Phi\big(x\big)\Big| \leq C n^{-1/2},
\end{align}
and an analogous result for the Wasserstein metric, subject only to mild weak dependence conditions. In particular, we only require a polynomial decay for our measure of dependence. Consequently, this also results in a polynomial decay for the autocovariance function $\gamma_X$ only, that is
\begin{align}
\gamma_{X}(h) \leq C |h|^{-\ad}, \quad \gamma_X(h) = \E X_h X_0, \quad h \in \Z,
\end{align}
with $C > 0$ and $\ad > 13/6$ (in our case, see Assumption \ref{ass_main_dependence} below and Lemma \ref{lem:bound:cov}). Our results apply to a large and diverse number of prominent dynamical systems and time series models used in econometrics, finance, physics and statistics. Among others, this includes random walks on the general linear group, products of positive random matrices, functionals of Garch models of any order, functionals of dynamical systems arising from stochastic differential equations (SDEs), functionals of infinite order Markov chains, linear processes and iterated random systems and many more, see Section \ref{sec:ex} for more details. In all those cases, it appears that this is the first time a result like \eqref{intro:eq:3} is established.

In addition, our results show another, interesting aspect of studentization, which is connected to the bias $|\sigma^2 - \E \hat{\sigma}_{nb}^2|$ of $\hat{\sigma}_{nb}^2$ and resolves a long-standing problem. In the classical case of estimating $\sigma^2$ (or, more generally, the spectral density), we have the (minimax) optimal rate
\begin{align}
\E \big|\sigma^2 - \hat{\sigma}_{nb}^2 \big|^2 \asymp n^{-\frac{2\bd}{2\bd+1}}, \quad \bd > 0,
\end{align}
provided that $\gamma_{X}(h) \asymp |h|^{-\bd - 1}$, $h \in \Z$, that is, provided we have polynomial decay, see for instance  \cite{bentkus:1985}. Here, the optimal rate originates from a bias-variance tradeoff. Optimal estimators are typically of the form
\begin{align}\label{defn:hat:sigma}
\hat{\sigma}_{nb}^2 = \sum_{|h| \leq b} \omega(h,b) \hat{\gamma}_X(h),
\end{align}
where $\omega$ is a weight function, $\hat{\gamma}_{X}(h)$ is the empirical autocovariance function, and $b$ is an appropriate bandwidth. An optimal estimator is given, for instance, by $\omega \equiv 1$ and $b \asymp n^{\frac{1}{2\bd+1}}$. More generally, one can employ flat-top kernels (cf. \cite{politis:2011}, \cite{politis:romano:1995}) to achieve such a result. Since $\bd$ is unknown in practice, adaptive estimators have to be constructed, which is a non-trivial task and requires heavy assumptions for rigorous results, e.g. \cite{andrews}, \cite{comte:2001:spectrum}, \cite{golubev:1993}, \cite{pickands:spectral:aos:1970}, \cite{newey:west:1994}. In the literature, it is often recommended to employ such optimal estimators for Student's statistic (that is, in \eqref{intro:eq:2}), see for instance \cite{newey:west:1994} or \cite{andrews} for very influential advocates\footnote{More than 3900 resp. 4900 citations according to Google-Scholar, 2022} and also \cite{andrews:monahan:1992}. But as it turns out, this is actually a bad choice and rules such as those presented in \cite{andrews}, \cite{andrews:monahan:1992} or \cite{newey:west:1994} should not be used in general, as was already pointed out for Gaussian time series in a related context (cf. \cite{goncalves_vogelsang_2011}, ~\cite{sun:phillips:econometrica:2008}): our results unveil the fact that
\begin{itemize}
  \item[(a)] optimal (minimax) estimators always yield suboptimal convergence rates in case of polynomial decay, that is, \eqref{intro:eq:3} does not hold,
  \item[(b)] simple oversmoothing yields optimal rates and thus \eqref{intro:eq:3}, regardless of polynomial, exponential or faster decay. In particular, {the level of oversmoothing is (asymptotically) irrelevant as long as $b$ is not excessively large. However, as pointed out by a reviewer, the choice of $b$ will certainly have an impact from a finite sample perspective.}
\end{itemize}

While there are reasons to suspect a negative result as in (a), item (b) is good news and not so obvious, at least on first sight. It shows that no complex, computationally expensive adaptive estimators are necessary to obtain the optimal rate in \eqref{intro:eq:3}. We note that this provides an example of a nonparametric statistical problem which does not suffer from the bias-variance dilemma. In fact, it turns out that the variance is completely irrelevant for the validity of \eqref{intro:eq:3}, possible choices of $b$ are entirely governed by the bias.

As mentioned above, it appears that previously only the case of exponential decay was studied in a more general context (cf. ~\cite{goetze_kunsch_1996, lahiri:2007:aos, lahirir:2010:aos}), where the situation is quite different from the polynomial case for the following reason. For $b = C \log n$, the exponential decay renders the bias part $\sum_{|h| > b} |\gamma_{X}(h)|$ irrelevant for $C > 0$ large enough, while maintaining the optimal rate to control the variance. Thus, in this case, there exist minimax optimal estimators that simultaneously lead to an optimal speed of convergence in \eqref{intro:eq:3}. On the other hand, it seems that both (a) and (b) previously have not been observed in the literature in the context of general, weakly dependent time series. Related discussions regarding bias problems in connection with bootstrapping or the special case of Gaussian time series are, however, present in the literature (cf. \cite{goncalves_vogelsang_2011}, \cite{goetze_kunsch_1996}, \cite{lahiri:2007:aos}, \cite{lahirir:2010:aos}, \cite{sun:phillips:econometrica:2008}, \cite{shao:X:aos:fixed:b:2013}), see also \cite{buhlmann2002bootstraps}, \cite{horowitz2019bootstrap}, \cite{kreiss2011bootstrap} regarding some general comments on the bootstrap. For more details, other notions of self-normalization and some more literature review, see the discussion after Corollary \ref{cor:equivalence}.

As already the i.i.d. case indicates, establishing \eqref{intro:eq:3} subject only to weak dependence conditions appears to be challenging, and the sparse attempts in the literature seem to confirm this. Our method of proof is based on a simple, yet effective linearization step, together with concentration inequalities for $\hat{\sigma}_{nb}^2$, variance expansions and Berry-Esseen bounds for weakly dependent sequences. Previous methods either rely on a direct expansion (cf. ~\cite{goetze_kunsch_1996, lahiri:2007:aos, lahirir:2010:aos}), or the transformation trick employed to obtain Cram\'{e}r-type deviation inequalities (cf. ~\cite{bentkus:goetze:aop:1996}, ~\cite{goetze:chistyakov:aop:2004}, ~\cite{shao_self_book_2009}, ~\cite{shao:jing:aop:2003}, ~\cite{shao:survey:2013}). Key point of our method is that it essentially allows us to obtain the (here simplified) approximation
\begin{align}\label{intro:approx}
\P\Big(\frac{X_1 + \ldots + X_n}{(n \hat{\sigma}_{nb}^2)^{1/2}} \leq x \Big) \approx \P\Big(\frac{Y_1 + \ldots + Y_n}{(n \sigma_b^2)^{1/2}}  \leq x \Big),
\end{align}
where $(Y_k)_{k \geq 1}$ is another, appropriately selected weakly dependent sequence and $\sigma_b^2 \approx \E \hat{\sigma}_{nb}^2$, removing the stochastic part in the denominator. The actual relation \eqref{intro:approx} is much more complicated though, and setting it into motion requires careful estimates. For more details on our approach and a comparison to the literature, see the beginning of Section \ref{sec:proofs}.

This work is structured as follows. Section \ref{sec:main} presents the main results, alongside comparisons and a brief discussion of the literature. Some examples are given in Section \ref{sec:ex}, where a divers mix of dynamical systems and time series is presented. Proofs are provided in Section \ref{sec:proofs}, where some relevant side results are relegated to Sections \ref{sec:proof:concentration} and \ref{sec:proof:var:exp}.

\section{Main results}\label{sec:main}

{\bf Notation:} For a random variable $X$, we write $\E X $ for expectation, $\|X\|_p$ for $\big(\E |X|^p \big)^{1/p}$, $p \geq 1$. In addition, $\lesssim$, $\gtrsim$, ($\asymp$) denote (two-sided) inequalities involving a multiplicative constant. For $a,b \in \R$, we put $a \vee b = \max\{a,b\}$, $a \wedge b = \min\{a,b\}$. We set $\sum_{i = a}^b (\cdot) = 0$ if $a>b$. We write $(\cdot)_{ab}$ for a double index if there is no confusion, but use a comma to separate, that is, $(\cdot)_{a,b}$, otherwise. Finally, for two random variables $X,Y$ we write $X \stackrel{d}{=} Y$ for equality in distribution.\\
\\
Over the past decades, a great number of different ways to define and quantify weak dependence have been established in the literature, see for instance \cite{bradley_book_vol1}, \cite{doukhan_2007_book} or \cite{wu_2005}. Our view point is the following. Let $(\varepsilon_k)_{k \in \Z}$ be a sequence of independent and identically distributed random variables taking values in some measurable space, and denote with $\mathcal{E}_k = \sigma( \varepsilon_j, \, j \leq k)$ the corresponding $\sigma$-algebra. We consider a sequence of real-valued random variables $(X_k)_{k \in \Z}$, where we always assume $X_k \in \mathcal{E}_k$, that is, we have the representation
\begin{align}\label{eq_structure_condition}
{X}_{k} = g_k\bigl(\varepsilon_{k}, \varepsilon_{k-1}, \ldots \bigr), \quad k \in \Z,
\end{align}
for some measurable functions $g_k$, and we sometimes abbreviate this with $X_k = g_k(\mathcal{E}_k)$. Over the past decades, an important question in the dynamical systems literature has been whether a (stationary) process $(X_k)_{k \in \Z}$ satisfies a representation like \eqref{eq_structure_condition} or not (e.g. \cite{emery_schachermayer_2001}, \cite{vershik_doc_transl}), and if so, whether the function $g_k$ depends on $k$ or not (e.g. \cite{feldman_rudolph_1998}). Both questions are, however, not relevant for our cause. It is well known (cf. \cite{rosenblatt:book}) that representation \eqref{eq_structure_condition} is always true for $1 \leq k \leq n$, $n$ finite\footnote{In this case $g_k$ can be selected as a map from $\R^k$ to $\R$, and this extends to Polish spaces.}, and, although we will always write and express our conditions in terms of $(X_k)_{k \in \Z}$ for simplicity, we effectively only work with $X_1,\ldots,X_n$. Similarly, we will always assume that $(X_k)_{k \in \Z}$ is strictly stationary, although this may be readily extended to local (weak) stationarity or quenched setups, see Example \ref{ex:randomwalk} for such a case.

A useful feature of representation \eqref{eq_structure_condition} is that it allows to give simple, yet efficient and general dependence conditions. Following \cite{wu_2005} and his notion of physical dependence, let $(\varepsilon_k')_{k \in \Z}$ be an independent copy of $(\varepsilon_k)_{k \in \Z}$ on the same, rich enough probability space. Slightly abusing notation, let
\begin{align}
X_k^{(l,')} = g_k\big(\varepsilon_k, \ldots,\varepsilon_{k-l+1}, \varepsilon_{k-l}', \mathcal{E}_{k-l-1}\big), \quad l \in \N, k \in \Z.
\end{align}
For $p \geq 1$, we then measure weak dependence in terms of the distance
\begin{align}\label{defn:dep:measure:general}
\theta_{lp} = \sup_{k \in \Z}\big\|X_k - X_k^{(l,')}\big\|_p, \quad l \in \N.
\end{align}

Observe that if the functions $g_k$ satisfy $g_k = g$, that is, they do not depend on $k$, the above simplifies to
\begin{align}\label{defn:dep:measure:bernoulli}
\theta_{lp} = \big\|{X}_{l} - {X}_{l}^{\prime}\big\|_p, \quad \text{with $X_l' = X_l^{(l,')}$ for $l \in \N$.}
\end{align}
In this case, the process $(X_k)_{k \in \Z}$ is typically referred to as (time homogenous) Bernoulli-shift process. As was pointed out by a reviewer, note that in general we may not even have $\theta_{lp} \to 0$ as $l$ increases, as can be seen from the simple example $X_k \equiv X$ for $k \in \Z$, with $\|X\|_p < \infty$.
In the sequel, we will require more than $\theta_{lp} \to 0$, namely a certain minimal amount of polynomial decay for $\theta_{lp}$ as $l$ increases, and express this as
\begin{align}\label{defn:Theta}
\Theta_{\ad p} = \big\|X_0\big\|_p +  \sum_{l = 1}^{\infty} l^{\ad} \theta_{lp}, \quad \ad \geq 0.
\end{align}

On the other hand, measuring (weak) dependence in terms of \eqref{defn:Theta}, that is, demanding $\Theta_{\ad p} < \infty$, is still quite general, easy to verify in many prominent cases, and has a long history going back at least to \cite{billingsley_1968}, \cite{ibraginov_1966}, we refer to Section \ref{sec:ex} for a brief account and some references. Among others, we specifically discuss the cases of random walks on the general linear group, functionals of Garch models of any order, functionals of dynamical systems arising from SDEs, functionals of linear processes and infinite order Markov chains.

Lastly, let us note that in some cases (e.g. \cite{cuny_dedecker_korepanov_merlevede_2019}, \cite{cuny_quickly_2020}), it is also worth to consider two-sided versions $X_k = g_k(\varepsilon_{j}, \, j \in \Z)$ instead of the (so-called) causal, one-sided representation \eqref{eq_structure_condition}. However, extending our results to this setting is beyond the scope and may prove to be challenging on a technical level.

Our main assumptions are now the following.

\begin{ass}\label{ass_main_dependence}
Let $(X_k)_{k\in \Z}$ with $\E X_k = 0$ be stationary, such that for $p \geq 6$, $\ad > 13/6$, we have
\begin{enumerate}
\item[\Atwo]\label{A2} $\Theta_{\ad p} = \|X_0\|_p + \sum_{l = 1}^{\infty}l^{\ad}\theta_{lp} < \infty$,
\item[\Athree]\label{A3} $\sigma^2 > 0$, where $\sigma^2 = \sum_{k \in \Z}\E X_0X_k$,
\item[\Afour]\label{A4} $b \to \infty$, such that $b \leq (n/\log^3 n)^{1/4}$.
\end{enumerate}
\end{ass}

\begin{rem}\label{rem:alt:b}
By strengthening the moment conditions, one may relax the constraint on $b$ a bit. More precisely,
if $p \geq 7$, then one may select $b \asymp n^{\bd}$ subject to $0 < \bd < p/(3p+2)$. We sketch the argument at the very end of the proof of Theorem \ref{thm:mein:be}.
\end{rem}


Let us briefly review Assumption \ref{ass_main_dependence}. Condition \hyperref[A3]{\Athree} is completely standard and requires no further discussion. \hyperref[A2]{\Atwo} is our weak dependence condition, where we emphasize that we only require the mild polynomial decay $\ad > 13/6$. As mentioned earlier, previous results all require
exponential decay among additional conditions, which is a much stronger requirement. On the other hand, note that condition $\ad > 13/6$ is slightly stronger compared to more recent results in the non-studentized case \eqref{intro:eq:1}, see for instance Lemma \ref{lem:BE:classic} below. This is due to the additional difficulty of studentization, and it is unclear whether this can be avoided.

Finally, \hyperref[A4]{\Afour} gives a range for the bandwidth $b$ regarding $\hat{\sigma}_{nb}^2$, defined in \eqref{defn:hat:sigma}. For a more detailed discussion on its influence on the results, see Corollary \ref{cor:equivalence} and Corollary \ref{cor:optimal:fails} and the attached comments.


Throughout this note, we set $\omega \equiv 1$ for the weights in \eqref{defn:hat:sigma}. For $0 \leq h \leq b$ (and $\hat{\gamma}_{X}(h) = \hat{\gamma}_{X}(-h)$), we then define the empirical autocovariance function $\hat{\gamma}_{X}$ as
\begin{align}\label{defn:gamma:X}
\hat{\gamma}_{X}(h) = \frac{1}{n} \sum_{k = h+1}^n (X_k - \bar{X}_n)(X_{k-h} - \bar{X}_n), \quad \bar{X}_n = \frac{1}{n}\sum_{k = 1}^n X_k,
\end{align}
where the normalisation $n^{-1}$ instead of $(n-h)^{-1}$ is just convenience and can be altered. We also denote with
\begin{align}\label{defn:sigma:nb}
\sigma_{b}^2 = \sum_{|h| \leq b} \E X_k X_0,
\end{align}
and hence $\sigma_{\infty}^2 = \sigma^2$.

\subsection{Kolmogorov distance}

Our main result is the following.

\begin{thm}\label{thm:mein:be}
Grant Assumption \ref{ass_main_dependence}. Then there exists a constant $C > 0$, only depending on $\Theta_{\ad 6}$ and $\sigma^2$, such that
\begin{align*}
\sup_{x \in \R}\Big|\P\Big(\frac{S_n}{(n \hat{\sigma}_{nb}^2)^{1/2}} \leq x\Big) - \Phi\Big( \frac{\sigma_{b}x}{\sigma}\Big)\Big| \leq C n^{-1/2},
\end{align*}
where $\sigma_{b}^2$ is defined in \eqref{defn:sigma:nb}.
\end{thm}

Observe that the approximating normal distribution has variance $\sigma^2/\sigma_{b}^2$. This is a key observation for applications, since the whole point of studentization is to render the approximating distribution pivotal, that is, it should not depend on any unknown quantities. Theorem \ref{thm:mein:be} together with the fact that
\begin{align}
\sup_{x \in \R}\Big|\Phi\Big( \frac{\sigma_{b}x}{\sigma}\Big) - \Phi\big(x\big) \Big| \asymp \big|\sigma^2 - \sigma_{b}^2\big|,
\end{align}
where constants only depend on $\sigma^2$, thus reveals the following bias problem.

\begin{cor}\label{cor:equivalence}
Grant Assumption \ref{ass_main_dependence}. Then the following statements are equivalent:
\begin{itemize}
  \item[(i)] There exists a constant $C > 0$, only depending on $\Theta_{\ad 6}$ and $\sigma^2$, such that
  \begin{align*}
  \big|\sigma^2 - \sigma_{b}^2 \big| \leq C n^{-1/2}.
  \end{align*}
  \item[(ii)] There exists a constant $C > 0$, only depending on $\Theta_{\ad 6}$ and $\sigma^2$, such that
  \begin{align*}
\sup_{x \in \R}\Big|\P\Big(\frac{S_n}{(n \hat{\sigma}_{nb}^2)^{1/2}} \leq x\Big) - \Phi\big(x\big)\Big| \leq C n^{-1/2}.
\end{align*}
\end{itemize}
\end{cor}

Corollary \ref{cor:equivalence} shows that the problem of optimal selection of $b$ (subject to Assumption \ref{ass_main_dependence}) entirely depends on the bias $\sigma^2 - \sigma_{b}^2$, whereas the variance of $\hat{\sigma}_{nb}^2$ is, somewhat surprisingly, completely irrelevant to obtain the optimal rate. This is in stark contrast to classical statistical problems,
where there is mostly some sort of trade-off, see for instance ~\cite{geman:1992}, p.45 and ~\cite{schmidt-hieber:bias-variance} for some more recent account. On the other hand, as mentioned earlier, it appears that for decades, bandwidth selection in this context has been (and still is) performed by minimizing the mean-squared error, we refer for instance to \cite{andrews}, \cite{andrews:monahan:1992}, \cite{newey:west:1994}, {\cite{robinson:eco:1998}}\footnote{He comments that such rules ``... are often based on considerations (such as minimum mean squared error ones) which are not obviously very relevant to the goal of satisfactorily studentization.``} or the more recent survey \cite{shao:X:survey:2015}. Apart from the above mentioned difficulty to actually find the optimal bandwidth in practice, Corollary \ref{cor:equivalence} shows that this is actually a bad choice, as the optimal convergence rate $n^{-1/2}$ can never be achieved in case of polynomial decay. We formulate this as the following corollary.

\begin{cor}\label{cor:optimal:fails}
Grant Assumption \ref{ass_main_dependence}, and suppose that 
\begin{align}\label{cor:optimal:fails:eq:1}
\big|\sum_{k \geq h} \E X_{k} X_0 \big| \asymp h^{-\bd}, \quad h \geq 1.
\end{align}
Then, employing the optimal estimator $\hat{\sigma}_{nb}^2$ with $b \asymp n^{1/(2\bd + 1)}$ leads to a convergence rate of $n^{-\bd/(2 \bd + 1)}$ in \eqref{intro:eq:3}, and thus always fails to achieve the optimal rate $n^{-1/2}$.
\end{cor}

Given this negative result, the question remains how to actually select $b$. The good news is that Corollary \ref{cor:equivalence} provides a simple answer: by Lemma \ref{lem:bound:cov}, we have $\sum_{k \geq b} |\E X_{k} X_0| \lesssim b^{-\ad + 1}$, hence any $b \gtrsim  n^{1/2(\ad-1)}$ results in a bias at the most of magnitude $n^{-1/2}$ and thus the optimal rate, provided that $b$ is not too large. Observe that for $\ad > 3$, selecting (for instance) $b \asymp n^{1/4}/\log n$ is within our constraint \hyperref[A4]{\Afour} (see also Remark \ref{rem:alt:b}), and satisfies $b \gtrsim n^{1/2(\ad-1)}$.
In other words, simple (sufficient) oversmoothing maintains the optimal rate of convergence. From a more general perspective, computations as in the proof of Proposition \ref{prop:variance:expansion} suggest that $b \lesssim n^{1/2}$ is always necessary, otherwise there will be another bias, distorting the optimal rate.\\
\\
In the literature, related bias problems have been reported in connection with bootstrap procedures in the presence of exponential decay, see \cite{goetze_kunsch_1996}, \cite{lahiri:2007:aos} and \cite{buhlmann2002bootstraps}, \cite{horowitz2019bootstrap}, \cite{kreiss2011bootstrap} for some general comments on the bootstrap. These problems are highly dependent on the actually used estimators (their weight functions) and their connection to the bootstrap method employed. In \cite{lahirir:2010:aos}, Remark 3.2, the bias problem in case of normal approximation and exponential decay is also discussed. It is mentioned in particular, that the usage of flat-top kernels (cf. \cite{politis:2011}, \cite{politis:romano:1995}) completely resolves it if $b \geq C \log n$, $C > 0$ sufficiently large. In the special case of Gaussian time series, results for studentized partial sums are also available in case of polynomial decay, accompanied by (significantly) better selection rules for $b$ also opposing those of \cite{andrews}, \cite{andrews:monahan:1992}, \cite{newey:west:1994}, see for instance \cite{goncalves_vogelsang_2011}, ~\cite{sun:phillips:econometrica:2008}, \cite{shao:X:aos:fixed:b:2013}. However, all these results crucially rely on the underlying Gaussianity (and the corresponding smoothness, linear structure and independence properties), and the results and most of the focus there is different since the third cumulant vanishes due to the Gaussianity. Much of the work in this particular area is motivated from a different usage of self-normalization, typically in connection with the so called ``fixed $b$-setup``, see the previous references and particularly \cite{shao:X:survey:2015} for a more general account. 

\subsection{Wasserstein distance}

For two probability measures $\P_1,\P_2$, let $\mathcal{L}(\P_1,\P_2)$ be the set of probability measures on $\R^2$ with marginals $\P_1, \P_2$. The Wasserstein metric (of order one) is defined as the minimal coupling $L^1$-distance, \hypertarget{Wone:eq10}{that} is,
\begin{align}\label{wasserstein_1}
W_1(\P_1,\P_2) = \inf\Big\{ \int_{\R} |x-y| \P(dx, dy): \, \P \in \mathcal{L}(\P_1, \P_2)\Big\},
\end{align}
see for instance \cite{villani2008optimal} for further properties.

For $\tau_n = C \sqrt{\log n}$, $C > 0$ sufficiently large, we consider the lower-truncated long-run variance estimator
\begin{align}\label{defn:sigma:truncated}
\hat{\sigma}^{\tau_n}_{nb} = \hat{\sigma}_{nb} \vee \tau_n^{-1}.
\end{align}
We note that other sequences could be used here as well, for instance $\tau_n = (\log n)^{1/2 + \delta}$, $\delta > 0$ or $\tau_n = n^{\delta}$, $\delta > 0$ sufficiently small (this depends on the underlying moments, see the proof for details). We now have the following result.
\begin{thm}\label{thm:mein:W}
Grant Assumption \ref{ass_main_dependence} with $p > 6$ and $\ad > 4$. Then there exists a constant $C > 0$, only depending on $\Theta_{\ad p}$, $p$ and $\sigma^2$, such that
\begin{align*}
W_1\Big(\P_{\frac{S_n}{n^{1/2} \hat{\sigma}_{nb}^{\tau_n}}}, \P_{G\frac{\sigma_{b}}{\sigma}}\Big)\leq C n^{-1/2},
\end{align*}
where $G$ is a standard Gaussian random variable and $\sigma_{b}^2$ is defined in \eqref{defn:sigma:nb}.
\end{thm}

Based on Theorem \ref{thm:mein:W}, one can formulate analogous results to Corollaries \ref{cor:equivalence} and \ref{cor:optimal:fails}. In addition, it also allows us to control the $L^q$-norm between distribution functions for any $q \geq 1$. This is illustrated by the following corollary, which is an immediate consequence of Theorems \ref{thm:mein:be} and \ref{thm:mein:W}.

\begin{cor}\label{cor:Lp}
Grant Assumption \ref{ass_main_dependence} with $p > 6$ and $\ad > 4$. Then there exists a constant $C > 0$, only depending on $\Theta_{\ad p}$, $p$ and $\sigma^2$, such that for any $q \geq 1$, we have
\begin{align*}
\int_{\R}\Big|\P\Big(\frac{S_n}{n^{1/2} \hat{\sigma}_{nb}^{\tau_n}} \leq x\Big) - \Phi\Big( \frac{\sigma_{b}x}{\sigma}\Big)\Big|^q dx \leq C n^{-q/2}.
\end{align*}
\end{cor}

Again, one can formulate analogous versions of Corollaries \ref{cor:equivalence} and \ref{cor:optimal:fails} with respect to the $L^q$-norm.

\section{Examples}\label{sec:ex}

As already previously mentioned, our setup contains a huge variety of prominent dynamical systems and time series models, see e.g. \cite{berkes2014}, \cite{wu_2011_asymptotic_theory} and \cite{wu_shao_iterated_2004} for an overview and further references. More recently, based on \cite{korepanov2018}, a connection was made
to many more (nonuniformly hyperbolic) dynamical systems in a series of works, see for instance \cite{cuny_dedecker_korepanov_merlevede_2019}, \cite{cuny_quickly_2020} and \cite{CUNY20181347}.

Below, we discuss some interesting examples in more detail, and also add a new class to the list of processes satisfying a condition like \hyperref[A2]{\Atwo}, namely systems arising from SDEs, see Example \ref{sec:ex:sde}. It seems that the latter connection was previously not observed in the literature. Moreover, it appears that Berry-Esseen bounds for studentized sums are entirely new for all those systems, there do not seem to be any comparable results in the literature.

Before discussing the actual examples, let us briefly touch on a useful property of our setup regarding functionals of the underlying sequence. To be more specific, let $(\mathbb{Y},\mathrm{d})$ be a metric space. In many cases, if $X_k = f(Y_k)$ for $Y_k \in \mathbb{Y}$ and $f: \mathbb{Y} \to \R$, it is easier to control $\mathrm{d}(Y_k,Y_k')$ rather than directly $|X_k-X_k'|$. Of course, this is only useful if the function $f$ is 'nice enough', allowing for a transfer of the rate. More generally, for any finite $d$, consider $\mathbb{Y}^{d}$ equipped with $\mathrm{d}(x,y) = \sum_{k = 1}^{d} \mathrm{d}(x_k,y_k)$, where $x = (x_1,\ldots,x_{d})$ for $x \in \mathbb{Y}^{d}$ and likewise for $y$. Let $f:\mathbb{Y}^{d} \to \R$ be a function satisfying
\begin{align}\label{ex:gen:hoelder:condition}
\big|f(x) - f(y)\big| \leq C_f \big(\mathrm{d}(x,y)^{\alpha} \wedge 1 \big) \big(1 + \mathrm{d}(x,0) + \mathrm{d}(y,0)\big)^{\beta},
\end{align}
with $C_f, \beta \geq 0$, $0 < \alpha \leq 1$ and $0 \in \mathbb{Y}^{d}$ some fixed point of reference. Define $X_k$ by
\begin{align}\label{ex:gen:g:function}
X_k = f\big(Y_k,Y_{k-1}, \ldots, Y_{k-d+1}\big) - \E f\big(Y_k,Y_{k-1}, \ldots, Y_{k-d+1}\big).
\end{align}
Note that for $\mathbb{Y} = \R$, this setup includes empirical autocovariance functions and other important statistics. If $q \geq 1 \vee p(\alpha + \beta)$ and $\E \mathrm{d}^q(Y_k,0) < \infty$, then straightforward computations reveal the following result.

\begin{prop}\label{prop:function}
Given \eqref{ex:gen:g:function}, there exists $C > 0$ such that
\begin{align}
\sup_{k \in \Z} \big\|X_k - X_k^{(l,')}\big\|_{p} \leq C \sup_{k \in \Z} \big( \E \mathrm{d}^{q}(Y_k, Y_k^{(l,')})\big)^{\alpha/q}.
\end{align}
\end{prop}

\begin{rem}
Observe that if $\sup_{k \in \Z} \|X_k - X_k^{(l,')}\|_{p}$ has exponential decay, one may select $\alpha > 0$ arbitrarily small and still maintain exponential decay.
\end{rem}

Armed with Proposition \ref{prop:function}, we are now ready to discuss some examples.

\subsection{Banach space valued linear processes}\label{ex:banach:linear}

Suppose that $\mathbb{Y} = \mathds{B}$ is a separable Banach space with norm $\|\cdot\|_{\mathds{B}}$. Slightly abusing notation, we write $\|X\|_p = \| \|X\|_{\mathds{B}}\|_p $ for the Orlicz-norm for a random variable $X \in \mathds{B}$. 
Let $(A_i)_{i \in \N}$ be a sequence of linear operators $A_i : \mathds{B} \to \mathds{B}$, and denote with
$\Vert A_i \Vert_{\mathrm{op}}$ the corresponding operator norm. For an i.i.d. sequence $(\varepsilon_k)_{k \in \Z} \in \mathds{B}$, consider the linear process
\begin{align*}
Y_k = \sum_{i = 0}^{\infty} A_i \varepsilon_{k-i}, \quad k \in \Z,
\end{align*}
which exists if $\|\varepsilon_0\|_{1} < \infty$ and $\sum_{i \in \N} \|A_i\|_{\mathrm{op}} < \infty$, which we assume from now on. Recall that autoregressive processes (even of infinite order) can typically be expressed as linear processes, in particular the (famous) dynamical system 2x mod 1 (Bernoulli convolution, doubling map). For the latter, we refer to Example 3.2 in ~\cite{jirak_be_aop_2016} and the references therein for more details. Obviously, we have the bound

\begin{align*}
\big\|Y_k - Y_k'\big\|_p \leq 2\big\|\varepsilon_k\big\|_p \big\|A_k\big\|_{\mathrm{op}}.
\end{align*}

Suppose that
\begin{align}\label{ex:linear:banach:condi}
\big\|\varepsilon_k\big\|_{q} < \infty, \quad \sum_{k = 0}^{\infty} k^{\ad} \big\|A_k\big\|_{\mathrm{op}}^{\alpha} < \infty, \quad \ad > 13/6,
\end{align}
for $q \geq p (\alpha + \beta)$, $p \geq 6$. We then obtain the following result.

\begin{cor}\label{cor:ex:lin}
Given the above conditions, let $X_k$ be as in \eqref{ex:gen:g:function}. Then \hyperref[A2]{\Atwo} holds. Hence, if $\sigma^2 > 0$ and $b \to \infty$ satisfies \hyperref[A4]{\Afour}, Theorem \ref{thm:mein:be} applies. For the validity of Theorem \ref{thm:mein:W}, we require $\ad > 4$ and $p > 6$.
\end{cor}

\subsection{SDEs}\label{sec:ex:sde}

Consider the following stochastic differential equation on $\mathbb{Y} = \R^d$ equipped with the Euclidian norm $\| \cdot \|_{\R^d}$:
\begin{align}\label{eq:ex:sde:1}
dY_t = a(Y_t) dt +\sqrt{2b(Y_t)} dB_t, \quad Y_0 = \xi,
\end{align}
where $(B_t)_{t\geq0}$ is a standard Brownian motion in $\R^d$, and the functions $a:\R^d \to \R^d$
and $b: \R^d \to \R^{d \times d}$ satisfy the following conditions. For a given matrix
$A$, we define the Hilbert–Schmidt norm $\|A\|_{\mathrm{HS}} = \sqrt{\operatorname{tr}(AA^{\top})}$.
We assume that the following stability condition (cf. \cite{Djellout:AoP:2004}) is satisfied: the functions $a$ and $b$ are Lipschitz continuous, and there exists $\gamma > 0$ such that
\begin{align}\label{ass:sde}
\|b(x) - b(y)\|_{\mathrm{HS}}^2 + \langle x-y,a(x)-a(y) \rangle \leq - \gamma\|x - y\|_{\R^d}^2, \quad x,y \in \R^d.
\end{align}
Regarding existence of a (strong and pathwise unique) solution, we refer to \cite{Djellout:AoP:2004} Section 4 or \cite{karatzas_shreve_1991}, Chapter 5. Now, for any fixed $\delta > 0$ and $t/\delta \in \N$, let
$\mathcal{I}_i = ((i-1)\delta, i\delta]$ for $i \geq 1$. We may thus write $Y_t$ as
\begin{align}\label{rep:sde}
Y_t = g_t\big((B_s - B_{t - \delta})_{s \in \mathcal{I}_{t/\delta}}, (B_s - B_{t - 2\delta})_{s \in \mathcal{I}_{t/\delta - 1}}, \ldots, (B_s)_{s \in \mathcal{I}_{1}}, \xi \big),
\end{align}
that is, as a map from $\big(C[0,\delta)\big)^{{t/\delta}} \times \R \to \R^d$, where $C[0,\delta)$ denotes the space of continuous functions mapping from $[0,\delta)$ to $\R$. By properties of the Brownian motion,
we can thus set $\varepsilon_i = (B_s - B_{(i-1)\delta})_{s \in \mathcal{I}_{i}}$.

For simplicity, we assume from now on that the initial condition $Y_0 = \xi$ admits a stationary solution $Y_t$. We then have the following result.

\begin{prop}\label{prop:sde}
Grant Assumption \ref{ass:sde}, and assume the (stationary) solution $(Y_t)_{t \geq 0}$ satisfies $\E \|Y_t\|_{\R^d}^q < \infty$, $q > 2$. Pick any $\delta > 0$. Then there exist $C, c > 0$, such that for any $2 \leq p < q$, we have
\begin{align*}
\sup_{t \in \Z}\E \big\|Y_t - Y_t^{(l,')}\big\|_{\R^d}^p \leq C \exp(-c l), \quad t/\delta \in \N.
\end{align*}
\end{prop}


Note that simple conditions for the existence of $\sup_{t \geq 0}\E \exp(-c \|Y_t\|_{\R^d}) < \infty$, $c > 0$, (also in the quenched case) are provided for instance in ~\cite{Djellout:AoP:2004}. Due to Proposition \ref{prop:sde}, we immediately obtain the following result.

\begin{cor}\label{cor:ex:sde}
Given the above conditions, pick any $\delta > 0$, let $X_k$ be as in \eqref{ex:gen:g:function}, and assume $\E \|Y_t\|_{\R^d}^q < \infty$, $q > p \beta$, $p \geq 6$. Then \hyperref[A2]{\Atwo} holds. Hence, if $\sigma^2 > 0$ and $b \to \infty$ satisfies \hyperref[A4]{\Afour}, Theorems \ref{thm:mein:be} and \ref{thm:mein:W} (with $p > 6$) apply.
\end{cor}

We note that discrete time analogues of \eqref{eq:ex:sde:1} naturally also meet our conditions, see for instance ~\cite{wu_2011_asymptotic_theory}.

\subsection{Left random walk on $GL_d(\R)$}\label{ex:randomwalk}

Cocycles, in particular the random walk on $GL_d(\R)$, have been heavily investigated in the literature, see e.g. \cite{bougerol_book_1985}, \cite{benoist2016}, \cite{benoist:quint:book:2016}, \cite{cuny2017}, \cite{CUNY20181347}, \cite{cuny:2023:aop} and \cite{ion:hui:jems:2022} for some more recent results. In this example, we will particularly exploit ideas of \cite{cuny2017}, \cite{CUNY20181347}. Let $(\varepsilon_k)_{k \in \N}$ be independent random matrices taking values in $G = GL_d(\R)$, with common distribution $\mu$. Let $A_0 = \mathrm{Id}$, and for every $n \in \N$, $A_n = \prod_{i = 1}^n \varepsilon_i$. Denote with $\|\cdot\|_{\R^d}$ the Euclidean norm on $\R^d$, and likewise $\|g\|_{\R^d} = \sup_{\|x\|_{\R^d} = 1} \|g x\|_{\R^d}$ the induced operator norm. We adopt the usual convention that $\mu$ has a moment of order $p$, if
\begin{align*}
\int_G \big(\log N(g) \big)^p \mu(d g) < \infty, \quad N(g) = \max\big\{\|g\|_{\R^d}, \|g^{-1}\|_{\R^d}\big\}.
\end{align*}
Let $\mathds{X} = P_{d-1}(\R^d)$ be the projective space of $\R^d\setminus\{0\}$, and write $\overline{x}$ for the projection from $\R^d\setminus\{0\}$ to $\mathds{X}$. We assume that $\mu$ is strongly irreducible and proximal, see ~\cite{cuny2017} for details. The left random walk of law $\mu$ started at $\overline{x} \in \mathds{X}$ is the Markov chain given by $Y_{0\overline{x}} = \overline{x}$, $Y_{k\overline{x}} = \varepsilon_k Y_{k-1\overline{x}}$ for $k \in \N$. Following the usual setup, we consider the associated random variables $(X_{k\overline{x}})_{k \in \N}$, given by
\begin{align}\label{ex:randomwalk:defn}
X_{k\overline{x}} = h\big(\varepsilon_k, Y_{k-1\overline{x}} \big), \quad h\big(g, {z}\big) = \log \frac{\| g z\|_{\R^d}}{\|z\|_{\R^d}},
\end{align}
for $g \in G$ and $z \in \R^d\setminus\{0\}$. It follows that, for any $x\in \mathbf{S}^{d-1}$, we have
\begin{align*}
S_{n\overline{x}} = \sum_{k = 1}^n \big(X_{k\overline{x}} - \E X_{k\overline{x}}\big) = \log \big\|A_n \overline{x}\big\|_{\R^d} - \E \log \big\|A_n \overline{x}\big\|_{\R^d}.
\end{align*}
Following ~\cite{CUNY20181347}, if $p > (2+\ad)q + 1$, then Proposition 3 in ~\cite{cuny2017} implies
\begin{align*}
\sum_{k = 1}^{\infty} k^{\ad} \sup_{\overline{x},\overline{y} \in \mathds{X}}\big\|X_{k\overline{x}} - X_{k\overline{y}} \big\|_q < \infty.
\end{align*}
In particular, it holds that
\begin{align*}
\lim_{n \to \infty} n^{-1} \E S_{n\overline{x}}^2 = \sigma^2,
\end{align*}
where the latter does not depend on $\overline{x} \in \mathds{X}$. We are now exactly in the \textit{quenched} setup briefly mentioned in the introduction, and can state the corollary below which follows from straightforward adjustments of Theorems \ref{thm:mein:be} and \ref{thm:mein:W}.

\begin{cor}\label{cor:randomwalk}
If $p > (2+\ad)q + 1$, $q \geq 6$, then for $\overline{x} \in \mathds{X}$, the process $X_{k\overline{x}}$ defined in \eqref{ex:randomwalk:defn} satisfies \hyperref[A2]{\Atwo} (quenched). Hence, if $\sigma^2 > 0$ and $b \to \infty$ satisfies \hyperref[A4]{\Afour}, Theorems \ref{thm:mein:be} and \ref{thm:mein:W} apply.
\end{cor}

\subsection{Iterative random systems}\label{sec:ex:iterative}

Let $(\mathbb{Y},\mathrm{d})$ be a complete, separable metric space. An iterated random
function system on the state space $\mathbb{Y}$ is defined as
\begin{align}
Y_k = F_{\varepsilon_k}\big(Y_{k-1} \big), \quad k \in \N,
\end{align}
where $\varepsilon_k \in \mathbb{S}$ are i.i.d. with $\varepsilon \stackrel{d}{=} \varepsilon_k$, where $\mathbb{S}$ is some measurable space. Here, $F_{\varepsilon}(\cdot) = F(\cdot, \varepsilon)$ is the $\varepsilon$-section of a jointly measurable function $F : \mathbb{Y} \times \mathbb{S} \to \mathbb{Y}$. Many dynamical systems, Markov processes and non-linear time series are within this framework, see for instance \cite{diaconis_freedman_1999}, \cite{wu_shao_iterated_2004}. For $y \in \mathbb{Y}$, let
$Y_k(y) = F_{\varepsilon_k} \circ F_{\varepsilon_{k-1}} \circ \ldots \circ F_{\varepsilon}(y)$, and, given $y,y' \in \mathbb{Y}$ and $\gamma > 0$, we say that the system is $\gamma$-\textit{moment contracting} if there exists $0 \in \mathbb{Y}$, such that for all $y \in \mathbb{Y}$ and $k \in \N$
\begin{align}\label{ex:iterated:contraction}
\E \mathrm{d}^{\gamma}\big(Y_k(y),Y_k(0)\big) \leq C \rho^k \mathrm{d}^{\gamma}(y,0), \quad \rho \in (0,1).
\end{align}
We note that slight variations exist in the literature. We now have the following result, which is an almost immediate consequence of Theorem 2 in \cite{wu_shao_iterated_2004}.

\begin{prop}\label{prop:iterated:1}
Assume that \eqref{ex:iterated:contraction} holds for some $\gamma > 0$, and that $\E \mathrm{d}^{\gamma}\big(0,Y_1(0)\big) < \infty$ for some $0 \in \mathbb{Y}$. Then there exists $C > 0$, such that
\begin{align*}
\sup_{k \in \Z}\E \mathrm{d}^{\gamma}\big(Y_k,Y_k^{(l,')}\big) \leq C \rho^l.
\end{align*}
If $\E \mathrm{d}^{s}\big(0,Y_1(0)\big)$ for $s > p \geq 1$ instead, then even
\begin{align*}
\sup_{k \in \Z}\E \mathrm{d}^p\big(Y_k,Y_k^{(l,')}\big) \leq C \rho^l.
\end{align*}
\end{prop}


As a consequence, we obtain the following corollary.

\begin{cor}\label{cor:iterated}
If $q > p \beta$, $p \geq 6$, then the process $X_{k}$ defined in \eqref{ex:gen:g:function} satisfies \hyperref[A2]{\Atwo}. Hence, if $\sigma^2 > 0$ and $b \to \infty$ satisfies \hyperref[A4]{\Afour}, Theorems \ref{thm:mein:be} and \ref{thm:mein:W} apply.
\end{cor}

\subsection{Products of positive random matrices}\label{sec:ex:Products:of:positive:random matrices}

Similar to Example \ref{ex:randomwalk}, products of positive random matrices have some history in the literature, see for instance \cite{hennion:1997}, \cite{hennion:herve:2004:aop} and \cite{ion:hui:jems:2022}, \cite{cuny:dedecker:merlevede:2023:positive:matrices} for some more recent results. For this example, we closely follow the discussion in \cite{hennion:herve:2004:aop} which allows us to use results from Example \ref{ex:iterated:contraction}. Let $G$ be the semigroup of $q \times q$ matrices with nonnegative entries which are allowable, namely, every row and every column contains a strictly positive element, and denote by $G^{\circ}$ the ideal of $G$ composed of matrices with strictly positive entries. For $g \in G$ and $w \in \R^q$, we denote by $g(w)$ the image of $w$ under $g$; the cone
\begin{align*}
\mathrm{C} = \{w:w = (w_1,\ldots,w_q) \in \R^q, w_k > 0, k = 1,\ldots,q\}
\end{align*}
is invariant under all $g \in G$. Define $M$ to be the intersection of the hyperplane
$\{w:w \in \R^q, \sum_{k=1}^q w_k = 1\}$ of $\R^q$ with $\mathrm{C}$. The linear space $\R^q$ is endowed with the $\ell_1$-norm $\|\cdot\|$ defined by
\begin{align*}
\|w\| = \sum_{k = 1}^q |w_k|, \quad w = (w_1,\ldots,w_q) \in \R^q,
\end{align*}
and for each $g \in G$, we set
\begin{align*}
\|g\| = \sup\{\|g(y)\| : y \in M\}, \,v(g) = \inf\{\|g(y)\| : y \in M\}.
\end{align*}
The semigroup $G$ being equipped with its Borel $\sigma$-field $\mathcal{G}$, we consider a probability distribution $\pi$ on $G$ for which there exists an integer $n_0$, such that the support of the random variable $R_{n_0}$ contains a matrix of $G^{\circ}$, where $R_n = \varepsilon_n \ldots \varepsilon_1$ and $\varepsilon_k \in G$ are i.i.d. with law $\pi$. This property is sometimes referred to as strictly contracting, see Definition 2.2 in \cite{cuny:dedecker:merlevede:2023:positive:matrices}. Denote by $g^{\ast}$ the adjoint of $g$, and let
\begin{align*}
\mathcal{L}^q = \int_G\big(|\log\|g\| + |\log v(g)| + |\log v(g^*)| \big)^q d \pi(g).
\end{align*}
For $y \in M$, we are now interested in $\log \|R_n(y)\| - \E \log \|R_n(y)\|$. According to \cite{hennion:herve:2004:aop}, this can be rewritten as
\begin{align}\label{eq:ex:product:positive:random}
\log \|R_n(y)\| - \E \log \|R_n(y)\| = \sum_{k = 1}^n f\big(\varepsilon_{k}, R_{k-1}y\big),
\end{align}
where $f$ (essentially) satisfies the conditions \eqref{ex:gen:hoelder:condition}, with $\alpha = \beta = d = 1$, see \cite{hennion:herve:2004:aop}, p.1944 for details. Moreover, the (normalised) system $R_{k-1}y$ may be realised as an iterative random systems, meeting the conditions of Proposition \ref{prop:iterated:1}, see \cite{hennion:herve:2004:aop}, p.1943. For detailed arguments as to how and why, we refer to \cite{cuny:dedecker:merlevede:2023:positive:matrices}. In particular, one could also directly use Proposition 3.2 therein to establish exponential decay for the coefficients $\theta_{kp}$.

As in Example \ref{ex:randomwalk}, we are now in the situation of a \textit{quenched} setup, and straightforward adjustments of Theorems \ref{thm:mein:be} and \ref{thm:mein:W} yield the following result.

\begin{cor}\label{cor:positive:random:functions}
Grant the above conditions, and suppose that $\mathcal{L}^{p} < \infty$, $p \geq 6$. If $\sigma^2 > 0$ and $b \to \infty$ satisfies \hyperref[A4]{\Afour}, Theorems \ref{thm:mein:be} and \ref{thm:mein:W} (with $p > 6$) apply to \eqref{eq:ex:product:positive:random}.
\end{cor}

\subsection{GARCH$(\mathfrak{p},\mathfrak{q})$ processes}\label{sec:ex:garch}

A very prominent stochastic recursion is the GARCH($\mathfrak{p},\mathfrak{q})$ sequence, given through the relations
\begin{align*}
&Y_k = \varepsilon_k L_{k} \quad \text{where $(\varepsilon_{k})_{k \in \Z}$ is a zero mean i.i.d. sequence and}\\
&L_k^2 = \mu + \alpha_1 L_{k - 1}^2 + \ldots + \alpha_\mathfrak{p} L_{k - \mathfrak{p}}^2 + \beta_1 Y_{k - 1}^2 + \ldots + \beta_{\mathfrak{q}} Y_{k - \mathfrak{q}}^2,
\end{align*}
with $\mu > 0$, $\alpha_1,\ldots,\alpha_\mathfrak{p}, \beta_1,\ldots,\beta_{\mathfrak{q}} \geq 0$. We refer to \cite{garch:book:2010} for some general aspects on Garch processes and their importance. Assume first $\|\varepsilon_k\|_q < \infty$ for some $q \geq 2$. A key quantity here is
\begin{align*}
\gamma_C = \sum_{i = 1}^{r} \bigl\|\alpha_i + \beta_i \varepsilon_i^2\bigr\|_{q/2}, \quad \text{with $r = \max\{\mathfrak{p},\mathfrak{q}\}$},
\end{align*}
where we replace possible undefined $\alpha_i, \beta_i$ with zero. If $\gamma_C < 1$, then $(Y_k)_{k \in \Z}$ is (strictly) stationary. In particular, one can show the representation
\begin{align*}
Y_k = \sqrt{\mu}\varepsilon_k\biggl(1 + \sum_{n = 1}^{\infty} \sum_{1 \leq l_1,\ldots,l_n\leq r} \prod_{i = 1}^n\bigl(\alpha_{l_i} + \beta_{l_i}\varepsilon_{k - l_1 - \ldots - l_i}^2\bigr) \biggr)^{1/2},
\end{align*}
we refer to ~\cite{aue_etal_2009} for comments and references. Using this representation and the fact that $|x-y|^q \leq |x^2 - y^2|^{q/2}$ for $x,y \geq 0$, one can follow the proof of Theorem 4.2 in ~\cite{aue_etal_2009} to show that
\begin{align*}
\bigl\|Y_k - Y_k'\bigr\|_q \leq C \rho^{k}, \quad \text{where $0 < \rho < 1$.}
\end{align*}

\begin{cor}\label{cor:ex:garch}
Given the above conditions, let $X_k$ be as in \eqref{ex:gen:g:function}, and assume $q > p \beta$, $p \geq 6$. Then \hyperref[A2]{\Atwo} holds. Hence, if $\sigma^2 > 0$ and $b \to \infty$ satisfies \hyperref[A4]{\Afour}, Theorems \ref{thm:mein:be} and \ref{thm:mein:W} (with $p > 6$) apply.
\end{cor}

We note that analogous results can be shown for augmented Garch processes, see ~\cite{berkes_2008_letter}, \cite{jirak:tams:2021}. Also note that from a more general perspective, Garch processes may be regarded as iterative random systems.

\subsection{Markov chains of infinite order}\label{ex:winterberger}



Let $\mathbb{S} = \mathds{B}$ be a Banach space with norm $\|\cdot\|_{\mathds{B}}$. Recall that $\|X\|_p = \| \|X\|_{\mathds{B}}\|_p $ denotes the Orlicz-norm for a random variable $X \in \mathds{B}$.
Consider a sequence $(a_k)_{k \geq 1} \in \R^+$ with $\sum_{k \geq 1} a_k < 1$, and let $(\varepsilon_k)_{k \in \Z} \in \mathds{B}$ be i.i.d. Let $F:\mathds{B}^{\N} \times \mathds{B} \to \mathds{B}$ be such that

\begin{align}
&\big\|F(x,\varepsilon_0) - F(y,\varepsilon_0) \big\|_{q} \leq \sum_{k = 1}^{\infty} a_k \|x_k - y_k\|_{\mathds{B}},\\
&\big\|F(0,0,\ldots,\varepsilon_0)\big\|_q < \infty, \quad x,y \in \mathds{B}^{\N}, 
\end{align}
where we write $x = (x_1,x_2,\ldots)$ for $x \in \mathds{B}^{\N}$ and $0 \in \mathds{B}$ is some point of reference. The Markov chain of infinite order is then (formally) defined as
\begin{align}
Y_k = F\big(Y_{k-1},Y_{k-2},Y_{k-3},\ldots, \varepsilon_k\big).
\end{align}

Existence and further properties are established in \cite{doukhan:winterberger:2008:spa}. In particular, if the sequence $(a_k)_{k \geq 1}$ satisfies
\begin{align}\label{ex:markov:infinite:weakdep:coef:decay}
\sum_{i \geq k} a_i \leq C_a k^{-\ad' - \delta}, \quad \delta > 0,
\end{align}
the results in \cite{doukhan:winterberger:2008:spa} imply that for some constant $C > 0$, we have
\begin{align}
\big\|Y_k - Y_k'\big\|_q \leq C  k^{-\ad'}.
\end{align}

Consequently, we have the following corollary.

\begin{cor}\label{cor:ex:markov;infinite}
Given the above conditions, let $X_k$ be as in \eqref{ex:gen:g:function} with $q \geq 1 \vee p(\alpha + \beta)$ and $\ad' > (\ad + 1)/\alpha$. Then \hyperref[A2]{\Atwo} holds. Hence, if $\sigma^2 > 0$ and $b \to \infty$ satisfies \hyperref[A4]{\Afour}, Theorems \ref{thm:mein:be} and \ref{thm:mein:W} apply with corresponding $\ad$ and $p$.
\end{cor}


\section{Proofs}\label{sec:proofs}

It appears that in the literature, two main methods have emerged to deal with studentization. The first one was used in \cite{bentkus_1997}, \cite{goetze_kunsch_1996}, \cite{lahiri:2007:aos}, \cite{lahirir:2010:aos} among others, and is simply a (Taylor) expansion around $\sigma$ and reads as

\begin{align}\label{defn:expansion}
\frac{S_n}{(n \hat{\sigma}_{nb}^2)^{1/2}} = \frac{S_n}{(n \sigma^2 )^{1/2}} + \frac{S_n (\sigma^2 - \hat{\sigma}_{nb}^2)}{2 (n \sigma^{3})^{1/2}} + \ldots\,\, .
\end{align}

The difficulty with this approach is that one has to additionally deal with the quadratic\footnote{or cubic, depends on the viewpoint.} term $S_n (\sigma^2 - \hat{\sigma}_{nb}^2)$, which is of magnitude $\sqrt{b}$ and thus not negligible. In the i.i.d. case, a more refined approach was highly successful. To briefly recall it, let $V_n^2 = \sum_{k = 1}^n X_k^2$ and $\hat{\sigma}_n^2 = \frac{n-1}{n}\sum_{k = 1}^n (X_k - \bar{X}_n)^2$ (slightly abusing notation). The trick is then to use the identity
\begin{align}\label{defn:iid:id}
\P\Big(S_n/\hat{\sigma}_n \geq x \Big) = \P\Big(S_n/V_n \geq x \sqrt{n/(n + x^2 -1)} \Big), \quad x \geq 0,
\end{align}
which has lead to a number of deep results, see for instance \cite{bentkus:goetze:aop:1996}, \cite{goetze:chistyakov:aop:2004}, \cite{hall:aop:1988}, \cite{shao:jing:aop:2003}, \cite{shao:survey:2013} or \cite{shao_self_book_2009}.

However, from a more general viewpoint that allows for dependence, \eqref{defn:iid:id} is again problematic since it heavily relies on the specific definition of $V_n$. In case of dependence, this has to be extended to block-type expressions, leading to many problems. Moreover, as useful as expression $S_n/V_n$ is in the independent case (dealing with it is still not easy), working with it directly in the presence of dependence is rather difficult and does not appear to have been so fruitful so far. However, by ``throwing away blocks``, \eqref{defn:iid:id} has been used in \cite{shao:aos:2022}, \cite{chen2016:aos}, employing a big block/small block technique, mimicking an i.i.d. setup. This resulted in sub-optimal rates, but maintained a Cram\'{e}r-type moderate deviation behaviour in the presence of exponentially decaying weak dependence.

As our starting point, we do not directly apply \eqref{defn:expansion} or \eqref{defn:iid:id}, but use the simple identity ($\sigma_{nb}^2$ is given in \eqref{defn:sigma:tilde}, but $\E \hat{\sigma}_{nb}^2 \approx \sigma_{nb}^2$)

\begin{align}\label{defn:THE} \nonumber
&\Big\{\frac{S_n}{(n \hat{\sigma}_{nb}^2)^{1/2}} \leq x\Big\} \\&= \Big\{\frac{S_n}{(n {\sigma}_{nb}^2)^{1/2}} - \frac{x(\hat{\sigma}_{nb}^2 - {\sigma}_{nb}^2)}{2 {\sigma}_{nb}^2} \leq x - x\frac{(\hat{\sigma}_{nb}^2 - {\sigma}_{nb}^2)^2}{2{\sigma}_{nb}^2(\hat{\sigma}_{nb} + {\sigma}_{nb})^2}\Big\},
\end{align}
$x \in \R$, which is obtained by multiplying with $\hat{\sigma}_{nb}$ and then expanding and dividing appropriately. Related expansions to \eqref{defn:THE} are present in the literature in connection with \eqref{defn:iid:id}, see e.g. \cite{shao_self_book_2009}. The key point here is that $x(\hat{\sigma}_{nb}^2 - {\sigma}_{nb}^2)$ is, in some sense, a linear term, while the quadratic term $(\hat{\sigma}_{nb}^2 - {\sigma}_{nb}^2)^2$ is of magnitude $|x| b/n$ this time, and hence should be negligible with high probability for $|x| b$ not too large. Our strategy of proof is thus to use \eqref{defn:THE}, and essentially show that
\begin{itemize}
\item[(A)] $|x|(\hat{\sigma}_{nb}^2 - {\sigma}_{nb}^2)^2$ is indeed negligible,
\item[(B)] $2 \sqrt{\sigma_{nb}}S_n/\sqrt{n} - x(\hat{\sigma}_{nb}^2 - {\sigma}_{nb}^2)$ satisfies a Berry-Esseen bound.
\end{itemize}
To set this into motion, we first require some additional notation and conventions.\\
\\
Throughout the proofs, constant $C > 0$ denotes a generic constant that may vary from line to line. Moreover, it only depends on $\sigma^2$, $p$ (one may set $p = 6$ in most cases though) and $\Theta_{\ad 6}$, exceptions are stated explicitly. Within proofs, we often use $a \lesssim b$ for such inequalities $a \leq C b$. To simplify notation, we also will assume $g_k = g$ in \eqref{eq_structure_condition}, hence \eqref{defn:dep:measure:bernoulli} applies. It will be also convenient to set $X_{k-h} = 0$ whenever $k-h \leq 0$. This only applies to this particular expression, not to $X_k$ or other indices, that is $X_i,X_k$ remain unchanged for $i,k \leq 0$.


Let $(\varepsilon_k')_{k \in \Z}$ be an independent copy of $(\varepsilon_k)_{k \in \Z}$. We denote with
\begin{align}\label{defn_couple_star}
X_k^{(l,*)} = g\big(\varepsilon_k,\ldots, \varepsilon_{k-l+1}, \varepsilon_{k-l}', \varepsilon_{k-l-1}', \ldots\big), \quad k \in \Z,
\end{align}
the coupled version, $X_k^* = X_k^{(k,*)}$ if $l = k$, and we measure the corresponding distance \hypertarget{lambdaqp:eq5}{with}
\begin{align}\label{defn:lambda:star}
\lambda_{kp}= \sup_{l \geq k}\big\|X_l-X_l^*\big\|_p, \quad \Lambda_{\mathfrak{c}p} = \big\|X_0\big\|_p +  \sum_{k = 1}^{\infty} k^{\mathfrak{c}} \lambda_{kp}, \, \mathfrak{c} \geq 0.
\end{align}
Using Lemma \ref{lem:theta:lambda:relation} in the Appendix, it follows that
\begin{align}\label{eq:theta:lambda:estimate:intro}
\lambda_{kp} \lesssim k^{-\ad} \Big(\sum_{j \geq k} j^{2\ad} \|X_j - X_j'\|_p^2 \Big)^{1/2} \lesssim \Theta_{\ad p} k^{-\ad}.
\end{align}
We will make frequent use of this estimate. Note that if we replace $\theta_{kp}$ with $\sup_{l \geq k} \theta_{lp}$ in \hyperref[A2]{\Atwo}, the above estimate can be improved to $\Theta_{\ad p} k^{-\ad - 1/2}$.

It will also be convenient to use the abbreviations (recall $X_{k-h} = 0$ for $k-h \leq 0$)
\begin{align}\label{defn:sigma:tilde}
\tilde{\sigma}_{nb}^2 = n^{-1} \sum_{k = 1}^n \Big(X_k^2 + 2 \sum_{h = 1}^b X_k X_{k-h}\Big), \quad \text{with} \quad {\sigma}_{nb}^2 = \E \tilde{\sigma}_{nb}^2,
\end{align}
and
\begin{align}\label{defn:Bkb:Ykb}
B_{kb} = \sum_{h = 1}^b X_{k - h}, \quad Y_k(x) = X_k\Big(1 - \frac{x}{2\sqrt{n}{\sigma}_{nb}}X_k - \frac{x}{\sqrt{n}{\sigma}_{nb}}B_{kb}\Big),
\end{align}
where we sometimes write $Y_k$ for short. Whenever we write $B_{kb}'$, $Y_k'(x)$, $B_{kb}^{\ast}$ or $Y_k^{\ast}(x)$ (and the same goes for other auxiliary random variables), we always consider them as functions $h_k(\varepsilon_k,\varepsilon_{k-1}, \ldots )$ for some function $h_k$, unless explicitly stated otherwise. For example, this means that (recall $X_{k-h} = 0$ for $k-h \leq 0$)
\begin{align*}
B_{kb}' = \sum_{h = 1}^b X_{k-h}^{(k-h,')} = \sum_{h = 1}^b X_{k-h}'.
\end{align*}
Also recall that some repeatedly used results from the literature are collected in Section \ref{sec:appendix}. Finally, note that some of the lemmas and propositions presented below can be shown subject to weaker conditions, but at the cost of lengthier and more technical proofs. Since none of these results represent a bottleneck regarding Assumption \ref{ass_main_dependence}, we refrain to do so.

\subsection{Kolmogorov distance}

We first establish a number of preliminary lemmas that we require for the proof.

\begin{lem}\label{lem:bound:difference:I}
Grant Assumption \ref{ass_main_dependence}. Then there exists a constant $C > 0$, such that for $k \geq 2b$
\begin{align*}
\big\|B_{kb} - B_{kb}' \big\|_p^2 \leq C k^{-2\ad +1}.
\end{align*}
\end{lem}

\begin{proof}[Proof of Lemma \ref{lem:bound:difference:I}]
Let $\mathcal{H}_l = \sigma(\mathcal{E}_l,\varepsilon_0')$, and denote with $\mathcal{P}_l(\cdot) = \E[\cdot|\mathcal{H}_l] - \E[\cdot|\mathcal{H}_{l-1}]$. By standard arguments, we have the representation
\begin{align}
B_{kb} - B_{kb}' = \sum_{l = 0}^{\infty} \mathcal{P}_{k-l}\big(B_{kb} - B_{kb}'\big).
\end{align}
Observe that by the triangle inequality, Jensen's inequality and stationarity, we have for $l \geq h$ the bound
\begin{align*}
\big\|\mathcal{P}_{k-l}\big(X_{k-h} - X_{k-h}'\big) \big\|_p &\leq 2 \big(\|X_{l-h} - X_{l-h}'\|_p \wedge \|X_{k-h} - X_{k-h}' \|_p \big),
\end{align*}
where we also used that $\mathcal{P}_{k-l}\big(X_{k-h}\big) \stackrel{d}{=}\mathcal{P}_{0}\big(X_{l-h}\big)$. If $l < h$, then trivially $\mathcal{P}_{k-l}\big(X_{k-h} - X_{k-h}'\big) = 0$. Since $k \geq 2b$, we thus obtain for any $l \geq 0$
\begin{align*}
&\big\|\mathcal{P}_{k-l}(B_{kb} - B_{kb}') \big\|_p \lesssim \sum_{h = 1}^{b} \|X_{k-h} - X_{k-h}' \|_p \lesssim (k-b)^{-\ad},
\end{align*}
where we also used \hyperref[A2]{\Atwo}. Similarly, for $l > b$, we get the bound $\|\mathcal{P}_{k-l}(B_{kb} - B_{kb}') \|_p \lesssim (l-b)^{-\ad}$. Employing Burkholder's inequality and the above, we conclude 
\begin{align*}
\Big\|B_{kb} - B_{kb}' \Big\|_p^2 &\lesssim \sum_{l = 0}^{\infty} \big\|\mathcal{P}_{k-l}\big(B_{kb} - B_{kb}'\big) \big\|_{p}^2 \lesssim k (k-b)^{-2\ad} +  \sum_{l > k} (l-b)^{-2\ad} \\&\lesssim (k-b)^{-2\ad +1} \lesssim k^{-2\ad + 1},
\end{align*}
where we also used $k \geq 2b$.
\end{proof}


{

\begin{lem}\label{lem:bound:Y:diff}
Grant Assumption \ref{ass_main_dependence}. Then there exists $C > 0$, such that for any $\ad^{\diamond} < \ad - 3/2$, we have for any $1 \leq q \leq p/2$
\begin{align*}
\sum_{k = 1}^{\infty} k^{\ad^{\diamond}} \sup_{l \geq k} \big\|Y_l(x) - Y_l'(x)\big\|_q \leq C\big(1 + n^{-1/2}|x| b^{\ad^{\diamond} + 1} \big).
\end{align*}
\end{lem}

\begin{proof}[Proof of Lemma \ref{lem:bound:Y:diff}]
Throughout the proof, we can and will assume w.l.o.g. ${\sigma}_{nb}^2 = 1$, since it is bounded away from zero uniformly in $n$ and $b$ (large enough), see \eqref{var:lower:bound}. By the triangle inequality
\begin{align*}
\big\|Y_k(x) - Y_k'(x)\big\|_q &\lesssim \big\|X_k - X_k'\big\|_q + n^{-1/2}|x| \big\|X_k^2 - (X_k^2)'\big\|_q \\&+ n^{-1/2}|x|\big\| X_k B_{kb} - X_k' B_{kb}'\big\|_q.
\end{align*}

Using $a^2-b^2 = (a-b)(a+b)$, the Cauchy-Schwarz inequality, the triangle inequality and stationarity, we have
\begin{align*}
\big\|X_k^2 - (X_k^2)'\big\|_q &\leq \big(\|X_k\|_{2q} + \|X_k'\|_{2q}\big)\big\|X_k - X_k'\big\|_{2q} \\&\leq 2 \big\|X_k\big\|_{2q} \big\|X_k - X_k'\big\|_{2q}.
\end{align*}

Next, note that due to Lemma \ref{lem:sipwu} in the Appendix, we have the bound
\begin{align}
\big\|B_{kb}\big\|_{2q} \lesssim \sqrt{b}.
\end{align}
Then, using the Cauchy-Schwarz and the triangle inequality, we obtain
\begin{align*}
\big\| X_k B_{kb} - X_k'B_{kb}'\Big\|_q &\leq \big\|X_k-X_k'\big\|_{2q} \big\|B_{kb}\big\|_{2q} + \big\|X_k'\big\|_{2q} \big\|B_{kb} - B_{kb}'\big\|_{2q} \\&\lesssim \sqrt{b} \big\|X_k-X_k'\big\|_{2q} + \big\|B_{kb} - B_{kb}'\big\|_{2q}.
\end{align*}
Assumption \hyperref[A2]{\Atwo} now implies $\sup_{k \geq 1} k^{\ad} \big\|X_k-X_k'\big\|_{2q} < \infty$, and the triangle inequality yields $\|B_{kb} - B_{kb}'\|_{2q} \leq  \Theta_{0,2q}$. 
Hence, by the above, \hyperref[A2]{\Atwo} and Lemma \ref{lem:bound:difference:I}, we get
\begin{align*}
&\sum_{k = 1}^{\infty} k^{\ad^{\diamond}} \sup_{l \geq k}\big\| X_l B_{lb} - X_l'B_{lb}'\Big\|_q \\&\lesssim \sqrt{b} \sum_{k = 1}^{\infty} k^{\ad^{\diamond} - \ad}  + \sum_{k = 1}^{2b} k^{\ad^{\diamond}} \sup_{l \geq k}\big\|B_{lb} - B_{lb}'\big\|_{2q} + \sum_{k > 2b} k^{\ad^\diamond} \sup_{l \geq k}\big\|B_{lb} - B_{lb}'\big\|_{2q} \\&\lesssim \sqrt{b} + \sum_{k = 1}^{2b} k^{\ad^{\diamond}}  + b^{\ad^{\diamond} - \ad + 3/2 } \lesssim  b^{1 + \ad^{\diamond}}.
\end{align*}
Combining all bounds, the claim now follows from the triangle inequality and \hyperref[A2]{\Atwo}.
\end{proof}

The following result is well-known. Since the argument is short, we provide it any way.

\begin{lem}\label{lem:bound:cov}
Grant Assumption \ref{ass_main_dependence} (for any $\mathfrak{a} > 0$). Then there exists an absolute constant $C > 0$, such that for any $h \geq 1$
\begin{align*}
\big|\E X_h X_0 \big| \leq C h^{-\mathfrak{a}}.
\end{align*}
\end{lem}

\begin{rem}\label{rem:cov:bound}
If we demand $\sum_{k \geq 1} k^{\ad} \sup_{l \geq k} \theta_{l2} < \infty$ in Assumption \ref{ass_main_dependence} instead of \hyperref[A2]{\Atwo}, then the estimate can be improved to $h^{-\mathfrak{a} - 1}$. 
\end{rem}

\begin{proof}[Proof of Lemma \ref{lem:bound:cov}]
Let $\mathcal{P}_l(\cdot) = \E[\cdot|\mathcal{E}_l] - \E[\cdot|\mathcal{E}_{l-1}]$. By standard arguments, we have the representation $X_k = \sum_{i = - \infty}^k \mathcal{P}_i(X_k)$ for any $k \in \Z$. The orthogonality of the martingale increments then implies
\begin{align*}
\big|\E X_h X_0 \big| = \Big|\sum_{i = - \infty}^0 \E \mathcal{P}_i\big(X_h\big) \mathcal{P}_i\big(X_0\big)\Big|.
\end{align*}
An application of the Cauchy-Schwarz inequality and the stationarity of $X_k$ further yields
\begin{align*}
\big|\E X_h X_0 \big| \leq \sum_{i = -\infty}^0 \big\| \mathcal{P}_i(X_h) \big\|_2\big\| \mathcal{P}_i(X_0) \big\|_2 \leq \sum_{i = 0}^{\infty} \big\|X_i-X_i'\big\|_2 \big\|X_{i+h} - X_{i+h}'\big\|_2.
\end{align*}
By Assumption \hyperref[A2]{\Atwo}, the above is bounded by $\lesssim \Theta_{02} h^{-\mathfrak{a}}$.
\end{proof}

Finally, we require the following (straightforward) adaptation of Theorem 2.3 in \cite{jirak2022BJ:improve}.

\begin{lem}\label{lem:BE:classic}
Let $X_{1},\ldots,X_n$ be a strictly stationary sequence with $X_k \in \mathcal{E}_k$.
Suppose there exist absolute constants $c^{\dag}, C^{\dag} > 0$,  such that for $q \geq 3$
\begin{itemize}
  \item[(i)] $\|X_k\|_q \leq C^{\dag}$, $\E X_k = 0$,
  \item[(ii)] $\sum_{1 \leq k \leq n} k^{\ad^{\dag}} \sup_{l \geq k}\|X_l-X_l'\|_q \leq C^{\dag}$ for $\ad^{\dag} > \frac{1}{2} + \frac{1}{2q}$,
  \item[(iii)] $\sum_{|k| \leq n} \E X_0 X_k \geq c^{\dag}$.
\end{itemize}
Then there exists a constant $C^{\ddag}$, only depending on $c^{\dag}$, $C^{\dag}$, such that
\begin{align*}
\sup_{x \in \R}\Big|\P\big(S_n \leq \|S_n\|_2 x \big) - \Phi\big(x\big)\Big| \leq \frac{C^{\ddag}}{\sqrt{n}}.
\end{align*}
\end{lem}

We are now ready to proceed to the proof of Theorem \ref{thm:mein:be}. As outlined at the beginning of Section \ref{sec:proofs}, we will show (A) and (B). This requires two additional (key) results that will be discussed separately: concentration bounds for the long-run variance (see Proposition \ref{prop:fn:longrun} in Section \ref{sec:proof:concentration}), and variance expansions (see Proposition \ref{prop:variance:expansion} in Section \ref{sec:proof:var:exp}). At first reading of the actual proof of Theorem \ref{thm:mein:be} below, it is recommended to only consider the statements of Proposition \ref{prop:fn:longrun} and Proposition \ref{prop:variance:expansion} (at the most) and skip their proofs, since these are lengthy.

\begin{proof}[Proof of Theorem \ref{thm:mein:be}]
We first establish control over the estimator $\hat{\sigma}_{nb}^2$. Observe that since $b \to \infty$ by \hyperref[A4]{\Afour}, we have for large enough $b, n$ due to Lemma \ref{lem:bound:cov} and \hyperref[A3]{\Athree}
\begin{align}\label{var:lower:bound}
{\sigma}_{nb}^2 \geq \frac{1}{2} \sum_{k \in \Z} \E X_k X_0 = \frac{\sigma^2}{2}> 0.
\end{align}
Next, for $C > 0$ (large enough), introduce the events

\begin{align}
\mathcal{A}_1 &= \Big\{\big|\tilde{\sigma}_{nb}^2 - {\sigma}_{nb}^2 \big| \leq C \frac{b \sqrt{\log n}}{\sqrt{n}} \Big\}, \nonumber \\
\mathcal{A}_2 &= \Big\{\big|\tilde{\sigma}_{nb}^2 - \hat{\sigma}_{nb}^2 \big| \leq C \frac{b \log n}{n} \Big\}.
\end{align}

Due to Proposition \ref{prop:fn:longrun} below, we have the estimate $\P\big(\mathcal{A}_1^c \big)\leq b^{-1} n^{1 - p/4}$ for sufficiently large $C$. 
Next, we note that
\begin{align*}
n \hat{\sigma}_{nb}^2 &= \sum_{k = 1}^n X_k^2  + 2 \sum_{h = 1}^b \sum_{k = h+1}^n X_k X_{k-h} \qquad (=n \tilde{\sigma}_{nb}^2)
\\&+2 \bar{X}_n \sum_{h = 1}^b \sum_{k = 1}^h X_k  + 2 \bar{X}_n \sum_{h = 1}^b \sum_{k = n-h+1}^n X_k
-b(b+1) \bar{X}_n^2 -(2b+1)n\bar{X}_n^2.
\end{align*}

Using Lemma \ref{lem:fn} in the Appendix, we conclude that for some constant $C > 0$ sufficiently large
\begin{align*}
\P\Big(\big|\sum_{h = 1}^b \sum_{k = n-h+1}^n X_k \big| \vee \big|\sum_{h = 1}^b \sum_{k = 1}^h X_k \big| \geq C b \sqrt{b \log n} \Big) &\leq \frac{1}{2}n^{1 - p/2},
\end{align*}
and also
\begin{align*}
\P\Big(\sqrt{n}|\bar{X}_n| \geq C\sqrt{\log n}\Big) \leq \frac{1}{2}n^{1-p/2},
\end{align*}
hence by the union bound $\P\big(\mathcal{A}_2^c \big) \leq n^{1-p/2}$. Summarising, we have
\begin{align}\label{eq:bound:A:events}
\P\big(\mathcal{A}_1^c \big)\leq b^{-1} n^{1 - p/4}, \quad \P\big(\mathcal{A}_2^c \big) \leq n^{1-p/2}.
\end{align}

Next, we use \eqref{eq:bound:A:events} to demonstrate that we may restrict our attention to $|x| \lesssim \sqrt{\log n}$. Due to \eqref{var:lower:bound}, we have for large enough $b,n$, that on the event $\mathcal{A}_1 \cap \mathcal{A}_2$
\begin{align}\label{eq:variance:lower:bound}
\hat{\sigma}_{nb}^2 \geq {\sigma}_{nb}^2 - |\hat{\sigma}_{nb}^2 - \tilde{\sigma}_{nb}^2| - |\tilde{\sigma}_{nb}^2 - {\sigma}_{nb}^2| \geq \frac{\sigma^2}{2} > 0.
\end{align}
An application of Lemma \ref{lem:fn} in the Appendix thus implies that for $\tau_n = C \sqrt{\log n}$, $C > 0$ large enough, we have
\begin{align}\label{eq:truncation:bound}
\P\Big(\frac{|S_n|}{(n \hat{\sigma}_{nb}^2)^{1/2}} \geq \tau_n \Big) &\leq n^{1-p/2} + \P\big(\mathcal{A}_1^c\big) + \P\big(\mathcal{A}_2^c\big) \nonumber \\&\leq 2 n^{1-p/2} + b^{-1} n^{1 - p/4} \lesssim n^{-1/2},
\end{align}
using \eqref{eq:bound:A:events} and Assumption \ref{ass_main_dependence}. Having completed these initial steps, we are now ready to employ our linearization argument. To this end, using $a^2 - b^2 = (a-b)(a+b)$, we have
\begin{align*}
\hat{\sigma}_{nb} - {\sigma}_{nb} = \big(\hat{\sigma}_{nb}^2 - {\sigma}_{nb}^2\big)\Big(\frac{1}{2{\sigma}_{nb}} + \frac{{\sigma}_{nb}  - \hat{\sigma}_{nb}}{2{\sigma}_{nb} (\hat{\sigma}_{nb} + \sigma_{nb})}\Big),
\end{align*}
and hence obtain the linearization \eqref{defn:THE}, which we restate for the sake of readability
\begin{align*}
\Big\{\frac{S_n}{(n \hat{\sigma}_{nb}^2)^{1/2}} \leq x\Big\} = \Big\{\frac{S_n}{(n {\sigma}_{nb}^2)^{1/2}} - \frac{x(\hat{\sigma}_{nb}^2 - {\sigma}_{nb}^2)}{2 {\sigma}_{nb}^2} \leq x - x\frac{(\hat{\sigma}_{nb}^2 - {\sigma}_{nb}^2)^2}{2{\sigma}_{nb}^2(\hat{\sigma}_{nb} + {\sigma}_{nb})^2}\Big\}.
\end{align*}

It follows that on the event $\mathcal{A}_1 \cap \mathcal{A}_2$, for $C > 0$ large enough, we have the inclusions
\begin{align}\label{eq:set:relation} \nonumber
&\Big\{\frac{S_n}{(n {\sigma}_{nb}^2)^{1/2}} - \frac{x(\tilde{\sigma}_{nb}^2 - {\sigma}_{nb}^2)}{2 {\sigma}_{nb}^2} \leq x - C \frac{|x| b^2 \log n}{n} \Big\} \subseteq \Big\{\frac{S_n}{(n \hat{\sigma}_{nb}^2)^{1/2}} \leq x\Big\} \\&\subseteq \Big\{\frac{S_n}{(n {\sigma}_{nb}^2)^{1/2}} - \frac{x(\tilde{\sigma}_{nb}^2 - {\sigma}_{nb}^2)}{2 {\sigma}_{nb}^2}  \leq x + C \frac{|x| b^2 \log n}{n} \Big\},
\end{align}
where we also used \eqref{eq:variance:lower:bound} and \hyperref[A4]{\Afour}. 
Let
\begin{align}
S_n(x) = \sum_{k = 1}^n Y_k(x), \quad |x| \leq \tau_n,
\end{align}
where we recall that $Y_k(x)$ is defined in \eqref{defn:Bkb:Ykb}. 
Note that by Lemma \ref{lem:bound:cov}, we have
\begin{align*}
\E S_n^2 = n \sigma^2  - \sum_{k \in \Z} \big(n \wedge |k|\big)\E X_k X_0 = n \sigma^2 + O\big(n^{2-\ad}\big).
\end{align*}
Due to Proposition \ref{prop:variance:expansion} (variance expansion), we thus have for large enough $b$, $n$
\begin{align}\label{eq:var:decomposition} \nonumber
\E S_n^2(x) &= \E S_n^2 + O\big((1+x^2)\sqrt{n}\big)
\\&= n \sigma^2 + O\big(n^{2-\ad} + (1+x^2)\sqrt{n}\big) \\
& \geq n \sigma^2 - O\big((1+x^2)\sqrt{n}\big). \nonumber
\end{align}
As before, we postpone the proof of Proposition \ref{prop:variance:expansion} and relegate it to Section \ref{sec:proof:var:exp}. Employing \eqref{var:lower:bound}, it follows that
\begin{align}\label{varSn(x):lower:bound}
\inf_{|x| \leq \tau_n} n^{-1}\big\|S_n(x)\big\|_2^2  \geq \frac{\sigma^2}{2} > 0
\end{align}
for $b, n$ large enough. Due to \eqref{varSn(x):lower:bound} and Lemma \ref{lem:bound:Y:diff}, we are now in position to apply Lemma \ref{lem:BE:classic} to $S_n(x)$. To this end, recall that Lemma \ref{lem:bound:Y:diff} requires the restriction $\ad^{\diamond} < \ad -3/2$. Since $p \geq 6$, we may select $q = 3$ in Lemma \ref{lem:BE:classic}, leading in total to the constraint $2/3 < \ad^{\dag} < \ad -3/2$. Since $\ad > 13/6$ by assumption, $\ad^{\dag} = \ad^{\diamond} = \ad/2 - 5/12$ is a valid choice, and we may apply Lemma \ref{lem:BE:classic}. We do so for every fixed $|x| \leq  \tau_n$, yielding
\begin{align}\label{eq:normal:approx}
\sup_{|x| \leq  \tau_n}\Big|\P\Big(\frac{S_n(x)}{\sqrt{n}} \leq x + \mu \Big) - \Phi\Big( \frac{x + \mu}{ \sqrt{n^{-1}\E S_n^2(x)}}\Big)\Big| \leq \frac{C}{\sqrt{n}}
\end{align}
for any $\mu \in \R$, where we emphasize that $C$ only depends on $\Theta_{\ad p}$, $\sigma^2$, appearing in Assumption \ref{ass_main_dependence}.

By \eqref{eq:var:decomposition}, $\ad > 13/6$ and Taylor expansion, it follows that uniformly for $x \in \R$
\begin{align}\label{eq:normal:approx:II}
\Big|\Phi\Big( \frac{x + \mu}{ \sqrt{n^{-1}\E S_n^2(x)}}\Big) - \Phi\Big(\frac{x}{\sigma}\Big) \Big| \lesssim \frac{1 + |\mu|}{\sqrt{n}} +\big|\mu\big|.
\end{align}
Let $\mu = C \frac{|x| b^2 \sigma_b \log n}{n}$, where we note that $\mu \lesssim n^{-1/2}$ for $|x| \leq  \tau_n$ due to \hyperref[A4]{\Afour} (this yields the constraint in \hyperref[A4]{\Afour} on $b$). Then due to \eqref{eq:bound:A:events}, \eqref{eq:set:relation}, \eqref{eq:normal:approx}, \eqref{eq:normal:approx:II} and the fact that $\sigma_b^2 = \sigma_{nb}^2 + O(1/n)$ (by Lemma \ref{lem:bound:cov}), we obtain the bound
\begin{align*}
\P\Big(\frac{S_n}{(n \hat{\sigma}_{nb}^2)^{1/2}} \leq x\Big) &\leq \P\Big(\frac{S_n(x)}{(\sigma_{b}^2 n)^{1/2}} \leq x + C \frac{|x| b^2 \log n}{n} \Big) + \P\big(\mathcal{A}_1^c\big) + \P\big(\mathcal{A}_2^c\big)
\\&\leq \Phi\Big( \frac{\sigma_{b}x}{\sigma}\Big) + O\big(n^{-1/2}\big),
\end{align*}
uniformly for $|x| \leq  \tau_n$. In the same manner, we get a corresponding lower bound. Combining this with \eqref{eq:truncation:bound} yields
\begin{align*}
\sup_{x \in \R}\Big|\P\Big(\frac{S_n}{(n \hat{\sigma}_{nb}^2)^{1/2}} \leq x\Big) - \Phi\Big( \frac{\sigma_{b}x}{\sigma}\Big)\Big| \lesssim n^{-1/2},
\end{align*}
and the proof is complete.

Lastly, let us briefly outline how to adjust the proof for the validity of Remark \ref{rem:alt:b}. If $p \geq 7$, then we may alter $\mathcal{A}_1$ to
\begin{align*}
\mathcal{A}_1 &= \Big\{\big|\tilde{\sigma}_{nb}^2 - {\sigma}_{nb}^2 \big| \leq C \frac{\sqrt{b \log n}}{\sqrt{n}} \Big\}.
\end{align*}
Proposition \ref{prop:fn:longrun} then yields the bound $\P(\mathcal{A}_1^c) = o(n^{-1/2})$. The bottleneck for $b$ now arises from Lemma \ref{lem:bound:Y:diff} and Lemma \ref{lem:BE:classic}. Note that one may derive analogous conditions also for $p \in (6,7)$. However, these are less simple.
\end{proof}

\subsection{Concentration of the empirical long-run variance}\label{sec:proof:concentration}

The aim of this section is to show the following concentration bound for the empirical long-run variance.

\begin{prop}\label{prop:fn:longrun}
Grant Assumption \ref{ass_main_dependence}. Then there exists $C > 0$, such that for $x \geq 0$
\begin{align*}
\P\Big(\frac{\sqrt{n}}{\sqrt{b}}\Big|\tilde{\sigma}_{nb}^2 - \E \tilde{\sigma}_{nb}^2 \Big| \geq C x \Big) \leq \Big(\frac{b}{n}\Big)^{p/4-1} \frac{1}{x^{p/2}} + \exp\big(-x^2\big).
\end{align*}
\end{prop}


There is an extensive literature on various aspects on the limiting behaviour of quadratic forms, see for instance
\cite{hsing:wu:aop:2004}, \cite{shao:wu:aos:2007:spectral}. Most of them use an approximating sequence of independent (resp. $m$-dependent sequences). In the spirit of the proof of Theorem \ref{thm:mein:be}, we directly interpret the quadratic form as a weakly dependent sequence, and then apply Lemma \ref{lem:fn}\footnote{The proof of this result is, however, based on an approximation with independent random variables.} in the Appendix. With some extra work, one may slightly weaken condition \hyperref[A2]{\Atwo} in Proposition \ref{prop:fn:longrun}. However, since the present form is entirely sufficient for our cause (the bottleneck for $\ad$ arises from Lemma \ref{lem:bound:Y:diff} and Lemma \ref{lem:BE:classic}), we stick to the simpler proof.

For the proof, we require some preliminary lemmas.

\begin{lem}\label{lem:concentration:I}
Grant Assumption \ref{ass_main_dependence}. Then there exists $C > 0$, such that for $1 \leq q \leq p/2$ and $k > j$, we have
\begin{align*}
&\big\|\E[X_k B_{kb} | \mathcal{E}_j] - \E[X_k B_{kb}]\big\|_q \leq C (k-j)^{-\ad + 1} + C (k-j)^{-\ad} \sqrt{b}.
\end{align*}
\end{lem}

\begin{proof}[Proof of Lemma \ref{lem:concentration:I}]
We start with the observation that
\begin{align*}
\E\big[X_k B_{kb} | \mathcal{E}_j\big] = \E\big[X_k \sum_{i = 1}^{k-j-1} X_{k-i} | \mathcal{E}_j\big] + \E\big[X_k | \mathcal{E}_j\big] \sum_{i = k-j}^{b} X_{k-i}.
\end{align*}

For $m = \lfloor (k-j)/2 \rfloor \vee 1$, consider the decomposition
\begin{align*}
X_k \sum_{i = 1}^{k-j-1} X_{k-i} = X_k \sum_{i = 1}^{m} X_{k-i} + X_k \sum_{i = m + 1}^{k-j-1} X_{k-i}.
\end{align*}
Note that by independence, we have
\begin{align}\label{eq:lem:concentration:I:1}
\E\big[ X_k \sum_{i = 1}^{m} X_{k-i} \big] = \E\big[X_k^{(k-j,\ast)} \sum_{i = 1}^{m} X_{k-i}^{(k-i-j,\ast)}\big|\mathcal{E}_j \big].
\end{align}

By \eqref{eq:lem:concentration:I:1}, Jensen's, Cauchy-Schwarz, the triangle inequality and stationarity, we obtain the chain of inequalities

\begin{align*}
&\big\|\E\big[X_k \sum_{i = 1}^{m} X_{k-i} | \mathcal{E}_j\big] - \E\big[ X_k \sum_{i = 1}^{m} X_{k-i} \big] \big\|_q \\&\leq
\big\|\big(X_k - X_k^{(k-j,\ast)}\big) \sum_{i = 1}^{m} X_{k-i} \big\|_q +
\big\|X_k^{(k-j,\ast)} \sum_{i = 1}^{m} \big(X_{k-i} - X_{k-i}^{(k-i-j,\ast)}\big) \big\|_q\\&\lesssim
\big\|X_k - X_k^{(k-j,\ast)}\big\|_{2q} \big\|\sum_{i = 1}^{m} X_{k-i} \big\|_{2q} +
\big\|X_k\big\|_{2q} \sum_{i = 1}^{m} \big\|X_{k-i} - X_{k-i}^{(k-i-j,\ast)}\big\|_{2q}\\&\lesssim\big\|X_{k-j} - X_{k-j}^{\ast}\big\|_{2q} \sqrt{m} +  \sum_{i = 1}^m \big\|X_{k-i-j} - X_{k-i-j}^{\ast}\big\|_{2q}
\\&\lesssim (k-j)^{-\ad} \sqrt{m} + \sum_{i = 1}^m (k-j-i + 1)^{-\ad} \lesssim (k-j)^{-\ad + 1},
\end{align*}
where we also used Lemma \ref{lem:sipwu} in the Appendix and \eqref{eq:theta:lambda:estimate:intro}. Next, since $X_k^{(m,\ast)}$ is independent of $\mathcal{E}_{k-m}$ and $\mathcal{E}_j \subseteq \mathcal{E}_{k-m}$, we have
\begin{align}
\E\big[ X_k^{(m,\ast)} \sum_{i = m + 1}^{k-j-1} X_{k-i} \big| \mathcal{E}_j\big] =  \E\big[\sum_{i = m + 1}^{k-j-1} X_{k-i} \big| \mathcal{E}_j\big] \E X_k^{(m,\ast)} = 0.
\end{align}
Arguing similarly as before, it follows that
\begin{align*}
&\big\|\E\big[ X_k \sum_{i = m + 1}^{k-j-1} X_{k-i} \big| \mathcal{E}_j\big] - \E\big[X_k \sum_{i = m + 1}^{k-j-1} X_{k-i}\big] \big\|_q \\&\lesssim \sqrt{k-j - m} \big\|X_{m}^{\ast} - X_{m}\big\|_{2q} \lesssim (k-j)^{-\ad + 1/2},
\end{align*}
and we also obtain

\begin{align*}
\big\|\E\big[X_k \big| \mathcal{E}_j\big] \sum_{i = k-j}^{b} X_{k-i} \big\|_q & \leq \big\|\E[X_k | \mathcal{E}_j] \big\|_{2q} \big\|\sum_{i = k-j}^{b} X_{k-i} \big\|_{2q} \\&\lesssim \big\| X_{k-j}^{\ast} - X_{k-j}\big\|_{2q} \sqrt{b} \lesssim (k-j)^{-\ad} \sqrt{b}.
\end{align*}
Combining all bounds, the triangle inequality yields
\begin{align*}
&\big\|\E[X_k B_{kb} | \mathcal{E}_j] - \E[X_k B_{kb}]\big\|_q \lesssim (k-j)^{-\ad + 1} + (k-j)^{-\ad + 1/2} + (k-j)^{-\ad} \sqrt{b},
\end{align*}
and hence the claim.

\end{proof}

Below, it will be convenient to define the quantity
\begin{align*}
W_k = 2 b^{-1/2} X_k \sum_{h = 1}^b X_{k-h} = 2b^{-1/2} X_k B_{kb}.
\end{align*}

We then have the following bound.

\begin{lem}\label{lem:sum:Wk}
Grant Assumption \ref{ass_main_dependence}. Then there exists $C > 0$, such that for $1 \leq q \leq p/2$, we have
\begin{align*}
\Big\|\sum_{k = 1}^n \big(W_k - \E W_k\big) \Big\|_q \leq C \sqrt{n}.
\end{align*}
\end{lem}

\begin{proof}[Proof of Lemma \ref{lem:sum:Wk}]
Let $\mathcal{P}_l = \E[ \cdot | \mathcal{E}_l] - \E[\cdot | \mathcal{E}_{l-1}]$. We employ the martingale decomposition ($0 \leq l \leq n$)
\begin{align*}
M_l = \sum_{k = 1}^n \E\big[W_k - \E W_k \big|\mathcal{E}_l \big], \quad M_l - M_{l-1} = \mathcal{P}_l\big(\sum_{k = 1}^n W_k \big).
\end{align*}

First, note that Lemma \ref{lem:sipwu} in the Appendix implies $b^{-1/2}\|B_{kb}\|_{2q} < \infty$. Hence, an application of the Cauchy-Schwarz inequality yields
\begin{align}\label{eq:lem:sum:Wk:2} \nonumber
\big\|W_k - W_k'\big\|_q &\leq \big\|X_k'\big\|_{2q} b^{-1/2}\big\|B_{kb} - B_{kb}'\big\|_{2q} + \big\|X_k - X_k'\big\|_{2q} b^{-1/2}\big\|B_{kb}\big\|_{2q}\\&
\lesssim b^{-1/2}\big\|B_{kb} - B_{kb}'\big\|_{2q} + \big\|X_k - X_k'\big\|_{2q}.
\end{align}

For $k \geq 2b$, we have from Lemma \ref{lem:bound:difference:I} $\|B_{kb} - B_{kb}'\|_{2q} \lesssim k^{-\ad + 1/2}$, hence
by \hyperref[A2]{\Atwo}
\begin{align}\label{eq:lem:sum:Wk:3}
\big\|W_k - W_k'\big\|_q \lesssim  k^{-\mathfrak{a}}\big(b^{-1/2} k^{1/2}  +  1\big), \quad k \geq 2b,
\end{align}
and thus by Lemma \ref{lem:theta:lambda:relation} in the Appendix
\begin{align}\label{eq:lem:sum:Wk:3.1}
\big\|W_k - W_k^{*}\big\|_q^2 \lesssim k^{-2\mathfrak{a} +1} \big(b^{-1} k + 1\big), \quad k \geq 2b.
\end{align}

Consider now the decomposition
\begin{align*}
\big\|\E\big[\sum_{k = 1}^n (W_k - \E W_k)\big|\mathcal{E}_0\big] \big\|_q &\leq \sum_{1 \leq k < 2b} \big\|\E\big[W_k - \E W_k\big|\mathcal{E}_0\big] \big\|_q \\&+ \sum_{k \geq 2b} \big\|\E\big[W_k - \E W_k\big|\mathcal{E}_0\big] \big\|_q.
\end{align*}

By Lemma \ref{lem:concentration:I} (with $j = 0$), we have
\begin{align}\label{eq:lem:sum:Wk:4}
\sum_{1 \leq k < 2b} \big\|\E\big[(W_k &- \E W_k)|\mathcal{E}_0\big] \big\|_q \lesssim \frac{1}{\sqrt{b}}
\sum_{k = 1}^{\infty} \big(k^{-\ad + 1} + k^{-\ad }\sqrt{b}\big) \lesssim 1.
\end{align}
On the other hand, the bound in \eqref{eq:lem:sum:Wk:3.1} yields
\begin{align*}
\sum_{k \geq 2b} \big\|\E\big[(W_k - \E W_k)|\mathcal{E}_0\big] \big\|_q &\leq \sum_{k \geq 2b} k^{-\ad+1/2}\big(b^{-1/2}k^{1/2} +  1\big) \lesssim b^{-\ad + 3/2}.
\end{align*}
Summarizing, we have
\begin{align}\label{eq:lem:sum:Wk:4.5}
\big\|M_0\big\|_q = \big\|\E\big[\sum_{k = 1}^n (W_k - \E W_k)\big|\mathcal{E}_0\big] \big\|_q \lesssim 1.
\end{align}

Next, by the triangle inequality, we get that for $j \geq 0$
\begin{align*}
\Big\|\mathcal{P}_j\Big(\sum_{k = 1}^n (W_k - \E W_k ) \Big)\Big\|_q & \leq \sum_{k = j}^{2b+j-1} \big\|\mathcal{P}_j\big(W_k\big) \big\|_q + \sum_{k \geq 2b+ j} \big\|\mathcal{P}_j\big(W_k\big) \big\|_q.
\end{align*}

By stationarity and \eqref{eq:lem:sum:Wk:3}, we have
\begin{align*}
\sum_{k \geq 2b + j} \big\|\mathcal{P}_j\big(W_k\big) \big\|_q &\leq \sum_{k \geq 2b+ j} \big\|W_{k-j} - W_{k-j}'\big\|_q \lesssim  b^{-\ad + 1},
\end{align*}
which is uniform in $j \geq 0$. Moreover, from stationarity, the triangle inequality and \eqref{eq:lem:sum:Wk:4}, we deduce the (uniform in $j \geq 0$) upper bound
\begin{align*}
\sum_{k = j}^{2b+j-1} \big\|\mathcal{P}_j\big(W_k\big) \big\|_q & \lesssim \sum_{1 \leq k < 2b} \big\|\E\big[W_k - \E W_k|\mathcal{E}_0\big] \big\|_q  + \big\|W_1\big\|_q \lesssim 1,
\end{align*}
where $\big\|W_1\big\|_q \lesssim 1$ follows from a similar argument as in \eqref{eq:lem:sum:Wk:2}. All in all, we obtain the estimate
\begin{align}\label{eq:lem:sum:Wk:5}
\sup_{j \geq 0} \Big\|\mathcal{P}_j\Big(\sum_{k = 1}^n (W_k - \E W_k ) \Big)\Big\|_q \lesssim 1.
\end{align}

Finally, using $(a+b)^2 \leq 2a^2 + 2b^2$ and Burkholder's inequality, we get from \eqref{eq:lem:sum:Wk:4.5} and \eqref{eq:lem:sum:Wk:5}
\begin{align*}
\Big\|\sum_{k = 1}^n \big(W_k - \E W_k\big) \Big\|_q^2&\leq 2 \big\|M_n - M_0\big\|_q^2 + 2 \big\|M_0\big\|_q^2 \\&\lesssim \sum_{k = 1}^n \big\|(M_{k} - M_{k-1})^2\big\|_{q/2} +  \big\|M_0\big\|_q^2\\ &\lesssim n,
\end{align*}
and the proof is complete.
\end{proof}

Next, consider the index sets (the blocks)
\begin{align}
\mathcal{I}_{lb} = \big\{bl,\ldots, b(l-1)+1\big\}, \quad l \in \N,
\end{align}
and the corresponding block variables
\begin{align*}
V_{lb} = \sum_{k \in \mathcal{I}_{lb}} \big(b^{-1/2}X_k^2 + W_{k}\big),
\end{align*}
and the blocks of the independent innovations  $\zeta_l = \big(\varepsilon_{bl}, \ldots, \varepsilon_{(l-1)b+1} \big)$. Note that there exist functions $g_l$, $\tilde{g}_l$, such that
\begin{align}
V_{lb} = g_l\big(\zeta_l, \zeta_{l-1},\ldots \big) = \tilde{g}_l\big(\varepsilon_{lb}, \varepsilon_{lb-1},\ldots\big).
\end{align}
We now make the following convention: In the proofs of Lemma \ref{lem:Vlb} and Proposition \ref{prop:fn:longrun} below, $V_{lb}'$, $V_{lb}^{\ast}$ are always taken with respect to $(\zeta_j)_{j \in \Z}$. For all other involved random variables, we stick to the usual notation.
\begin{lem}\label{lem:Vlb}
Grant Assumption \ref{ass_main_dependence}. Then there exists $C > 0$, such that for $1 \leq q \leq p/2$ and $l \geq 3$
\begin{align*}
\big\|V_{lb} - V_{lb}'\big\|_q \leq C l^{-\ad + 1/2},
\end{align*}
and for any $l \geq 1$
\begin{align*}
\big\|V_{lb} - V_{lb}'\big\|_q \leq C.
\end{align*}
\end{lem}

\begin{proof}[Proof of Lemma \ref{lem:Vlb}]
Observe first that by using a telescoping sum, stationarity and the triangle inequality, one gets
\begin{align*}
\big\|X_k - X_k^*\big\|_p \leq \sum_{j \geq k} \big\|X_k - X_k'\big\|_p.
\end{align*}
While this bound is clearly inferior compared to Lemma \ref{lem:theta:lambda:relation} in the Appendix, the argument is useful for the blocks $V_{lb}$ and essentially leads to the same bound as an adapted version of Lemma \ref{lem:theta:lambda:relation} in the present context. Employing it together with the Cauchy-Schwarz inequality, we obtain
\begin{align*}
\sqrt{b}\big\|V_{lb} - V_{lb}^{'}\big\|_q &\leq \sum_{k \in \mathcal{I}_{lb}} \sum_{h = 0}^b\big(\|X_{k+h}^2 - (X_{k+h}^{'})^2\|_q + \sqrt{b}\|W_{k+h} - W_{k+h}'\big\|_q\big)\\&\lesssim \sum_{k \in \mathcal{I}_{lb}}\sum_{h = 0}^b\big\|X_{k+h}\big\|_{2q}\big\|X_{k+h} - X_{k+h}'\big\|_{2q} \\&+ \sum_{k \in \mathcal{I}_{lb}}\sum_{h = 0}^b\big\|X_{k+h} - X_{k+h}'\big\|_{2q} \big\|B_{k+h,b}\big\|_{2q} \\& + \sum_{k \in \mathcal{I}_{lb}}\sum_{h = 0}^b \big\|X_{k+h}\big\|_{2q} \big\|B_{k+h,b} - B_{k+h,b}'\big\|_{2q}.
\end{align*}
By \hyperref[A2]{\Atwo} and Lemma \ref{lem:sipwu} in the Appendix, we have
\begin{align*}
\sum_{k \in \mathcal{I}_{lb}}\sum_{h = 0}^b \Big( \big\|X_{k+h} - X_{k+h}'\big\|_{2q} &+ \big\|X_{k+h} - X_{k+h}'\big\|_{2q} \big\|B_{k+h,b}\big\|_{2q}\Big) \\&\lesssim b^2 \big((l-1)b\big)^{-\mathfrak{a}} + b^{5/2}\big((l-1)b\big)^{-\mathfrak{a}}.
\end{align*}
Moreover, Lemma \ref{lem:bound:difference:I} implies 
\begin{align*}
\sum_{k \in \mathcal{I}_{lb}}\sum_{h = 0}^b \big\|B_{k+h,b} - B_{k+h,b}'\big\|_{2q}  \lesssim b^2 \big((l-1)b\big)^{-\mathfrak{a}+1/2}.
\end{align*}
Piecing everything together, we obtain for $l \geq 3$ due to $\mathfrak{a} > 13/6$
\begin{align*}
\big\|V_{lb} - V_{lb}^{'}\big\|_q \lesssim l^{-\mathfrak{a}+1/2}.
\end{align*}
This proves the first claim. For the second claim, we note that by Lemma \ref{lem:sipwu} in the Appendix
\begin{align*}
\Big\|\sum_{k \in I_{lb}} \big(X_k^2 - \E X_k^2 \big)\Big\|_q \lesssim \sqrt{b}.
\end{align*}
In addition, since the cardinality of $\mathcal{I}_{bl}$ is (bounded by) $b$, Lemma \ref{lem:sum:Wk} yields
\begin{align*}
\Big\|\sum_{k \in \mathcal{I}_{lb}} \big(W_k - \E W_k \big)\Big\|_q \lesssim 1.
\end{align*}
The claim then follows from the triangle inequality.
\end{proof}

Having established all the necessary preliminary results, we are now ready to complete the proofs.

\begin{proof}[Proof of Proposition \ref{prop:fn:longrun}]
Let $q = p/2$, and assume first that $n/b \in \N$. Then
\begin{align*}
n b^{-1/2} \tilde{\sigma}_{nb}^2 = \sum_{l = 1}^{n/b} \sum_{k \in \mathcal{I}_{lb}} \big(b^{-1/2}X_k^2 +  W_k\big) = \sum_{l = 1}^{n/b} V_{lb}.
\end{align*}
Since $\ad > 13/6$ by Assumption \ref{ass_main_dependence}, it follows from Lemma \ref{lem:Vlb} that
\begin{align}\label{eq:prop:fn:longrun:1}
\sum_{l \geq m} \big\|V_{lb}-V_{lb}'\big\|_q \lesssim \sum_{l \geq m} l^{-\mathfrak{a} +1/2} \lesssim m^{-\alpha},
\end{align}
with $\alpha > 2/3$. On the other hand, Lemma \ref{lem:Vlb} also implies
\begin{align}\label{eq:prop:fn:longrun:2}
\sum_{l \leq 3} \big\|V_{lb}-V_{lb}'\big\|_q \lesssim 1.
\end{align}
Hence by the above, we obtain for $\alpha > 2/3$ that
\begin{align}
\sum_{l \geq 1} \big\|V_{lb}-V_{lb}'\big\|_q < \infty, \quad \sum_{l \geq m} \big\|V_{lb}-V_{lb}'\big\|_q \lesssim m^{-\alpha}.
\end{align}
An application of Lemma \ref{lem:fn} in the Appendix (with $p = 3$) to $\sum_{l = 1}^{n/b} V_{lb}$ then yields the claim. Suppose now that $n/b \not \in \N$, i.e., $n = m b + a$ with $0<a < b$. Then we have one additional smaller block $V_{la}$, but since $a < b$, it is obvious that the result persists.
\end{proof}

\subsection{Variance expansion}\label{sec:proof:var:exp}


Recall the definition of $\lambda_{kp}$, $\Lambda_{\mathfrak{c}p}$, given in \eqref{defn:lambda:star}, and recall that due to \eqref{eq:theta:lambda:estimate:intro}, we have $\lambda_{kp} \lesssim k^{-\ad}$. As will be apparent from the proof, we can and will assume w.l.o.g. that ${\sigma}_{nb}^2 = 1$ throughout this section. Moreover, only within this section, we \textit{no longer} make the convention $X_{k-h} = 0$ for $k-h < 0$. The objective is to show the following result.

\begin{prop}\label{prop:variance:expansion}
Grant Assumption \ref{ass_main_dependence}, where we only demand $b \leq c n^{1/3}$, $c > 0$. Then there exists $C > 0$, such that for any $x \in \R$
\begin{align*}
\sum_{k \in \Z} \Big| \E X_k X_0 - \E \big(Y_k(x)-\E Y_k(x) \big) \big(Y_0(x) - \E Y_0(x)\big) \Big| \leq C \big(1+x^2\big) n^{-1/2}.
\end{align*}
Constant $C$ only depends on $\Lambda_{1,4}$ and $c$ (but also on $\sigma^2$ in the general case $\sigma_{nb}^2 \neq 1$).
\end{prop}

For the proof, we require the following preliminary result.
\begin{lem}\label{lem:triple:bound}
For $i \leq j \leq k$, we have
\begin{align*}
\big|\E X_i X_j X_k\big| \leq \|X_0\|_3^2 \big((\lambda_{k-i,3} + \lambda_{j-i,3}) \wedge \lambda_{k-j,3}\big).
\end{align*}
\end{lem}

\begin{proof}[Proof of Lemma \ref{lem:triple:bound}]
Due to $i \leq j \leq k$, we have
\begin{align*}
\E X_j X_k  = \E\big[(X_j X_k)^{(k-i,*)}\big|\mathcal{E}_i\big] = \E \big[X_j^{(j-i,*)} X_k^{(k-i,*)}\big|\mathcal{E}_i\big].
\end{align*}
Then, since $\E X_i = 0$, the triangle and H\"{o}lder's inequality yield
\begin{align*}
\big|\E X_i X_j X_k\big| &= \big|\E \big[ X_i \E\big[X_j X_k\big|\mathcal{E}_i\big]\big]\big| = \big|\E \big[ X_i \E\big[X_j X_k - (X_j X_k)^{(k-i,*)}\big|\mathcal{E}_i\big]\big]\big| \\&\leq \|X_i\|_3 \big(\|X_j\|_3 \|X_k -  X_k^{(k-i,*)}\|_3 + \|X_k\|_3 \|X_j -  X_j^{(j-i,*)}\|_3\big)\\&\leq \|X_0\|_3^2 \big(\lambda_{k-i,3} + \lambda_{j-i,3}\big).
\end{align*}
Similarly, one derives that
\begin{align*}
\big|\E X_i X_j X_k \big| \leq \|X_i\|_3 \|X_j\|_3 \|X_k - X_k^{(k-j,*)}\|_3  \leq \|X_0\|_3^2 \lambda_{k-j,3}.
\end{align*}
\end{proof}

\begin{proof}[Proof of Proposition \ref{prop:variance:expansion}]
We first make the expansion
\begin{align*}
 \E Y_k Y_0 &=  \E  X_k\Big(1 - \frac{x}{2\sqrt{n}}X_k - \frac{x}{\sqrt{n}}\sum_{h = 1}^b X_{k-h}\Big) X_0\Big(1 - \frac{x}{2\sqrt{n}}X_0 - \frac{x}{\sqrt{n}}\sum_{h = 1}^b X_{-h}\Big) \\&=  \E X_k X_0 - x 2^{-1} n^{-1/2} \Big( \E X_k^2 X_0 +  \E X_k X_0^2 \Big) \\&+ x^2 (4n)^{-1}  \E X_k^2 X_0^2 - x n^{-1/2} \Big(\E X_k X_0 B_{0b} + \E X_k B_{kb} X_0 \Big)  \\& + x^2 (2n)^{-1} \Big(\E X_k^2 X_0 B_{0b} +  \E X_k B_{kb} X_0^2 \Big) + x^2 n^{-1}  \E X_k B_{kb} X_0 B_{0b}.
\end{align*}
The related one for $\E Y_k \E Y_0$ is done in an analogous manner. We will now treat all terms separately, appropriately centred. To this end, recall that for any $q \geq 2$, we have
\begin{align}\label{eq:square:star}
\big\|X_k^2 - (X_k^2)^{\ast}\big\|_{q/2} \lesssim  \lambda_{kq}.
\end{align}

{\bf Term $\E X_k^2 X_0 + \E X_k X_0^2$:} Since $X_k^{\ast}$ and $X_0$ are independent, we have due to $\E X_0 = 0$, Cauchy-Schwarz and \eqref{eq:square:star}

\begin{align*}
\big|\E X_k^2 X_0 \big| &\leq \E \big|X_k^2 X_0 - (X_k^{\ast})^2 X_0 \big| \\&\leq \big\|X_k^2 - (X_k^2)^{\ast}\big\|_{2} \big\|X_0\big\|_2 \lesssim \lambda_{k4}.
\end{align*}

Similarly, one derives
\begin{align*}
\big|\E X_k X_0^2 \big| \lesssim \lambda_{k,2} \leq \lambda_{k4}.
\end{align*}
Hence, combining both bounds, we obtain
\begin{align}
\sum_{k \in \N}\big(|\E X_k^2 X_0| + |\E X_k X_0^2| \big) < \infty.
\end{align}

{\bf Term $\E X_k^2 X_0^2 - \E X_k^2 \E X_0^2$:}
By independence, $\E (X_k^{\ast})^2 X_0^2 = \E X_k^2 \E X_0^2$. Then by Cauchy-Schwarz and \eqref{eq:square:star}
\begin{align*}
\big|\E X_k^2 X_0^2 - \E X_k^2 \E X_0^2 \big| &\leq \big\|X_k^2 - (X_k^{\ast})^2\big\|_{2} \big\|X_0^2\big\|_2 \lesssim \lambda_{k4}.
\end{align*}
Hence
\begin{align}
\sum_{k \in \N}\big|\E X_k^2 X_0^2 - \E X_k^2 \E X_0^2\big| < \infty.
\end{align}

{\bf Terms $\E X_k B_{kb} X_0$ and $\E X_k X_0 B_{0b}$:}

Using Lemma \ref{lem:triple:bound}, it follows that
\begin{align*}
\big|\E X_k B_{kb} X_0\big| &\leq \sum_{h = 1}^{k/2} \big|\E X_k X_{k-h} X_0 \big| + \sum_{h = k/2 + 1}^k \big|\E X_k X_{k-h} X_0 \big|
\\&+ \sum_{h = k + 1}^{2k} \big|\E X_k X_{k-h} X_0 \big| + \sum_{h = 2k + 1}^b \big|\E X_k X_{k-h} X_0 \big| \\&\lesssim \sum_{h = 1}^{k/2} \lambda_{k-h,3} + \sum_{h = k/2 + 1}^k \lambda_{h3} +  \sum_{h = k+1}^{2k} \lambda_{h3} + \sum_{h = 2k + 1}^b \lambda_{h-k,3} \\&\lesssim \sum_{l \geq k/2} \lambda_{l3} + \sum_{l \geq k} \lambda_{l3} \lesssim k^{-\ad + 1}.
\end{align*}
Moreover, we have
\begin{align*}
\big|\E X_k X_0 B_{0,b}\big| &\leq \sum_{h = 1}^{k}\big| \E X_k X_0 X_{-h} \big| + \sum_{h = k+1}^{b}\big| \E X_k X_0 X_{-h} \big|\\
&\lesssim k \lambda_{k3} + \sum_{h = k+1}^{b} \lambda_{h3} \lesssim k^{-\ad + 1}.
\end{align*}

Combining both bounds, we deduce
\begin{align}\label{prop:variance:expansion:eq:4}
\sum_{k \in \N} \big(|\E X_k B_{kb} X_0| + |\E X_k X_0 B_{0b}| \big) < \infty.
\end{align}

{\bf Term $\E X_k^2 X_0 B_{0b} - \E X_k^2 \E X_0 B_{0b}$:} Note that

\begin{align*}
\E X_k^2 X_0 B_{0b} - \E X_k^2 \E X_0 B_{0b} = \E (X_k^2 - \E X_k^2) X_0 B_{0b}.
\end{align*}
Due to \eqref{eq:square:star}, we may argue as before in \eqref{prop:variance:expansion:eq:4} (Lemma \ref{lem:triple:bound} remains valid) to establish
\begin{align}
\sum_{k \in \N}\big|\E (X_k^2 - \E X_k^2) X_0 B_{0b}\big| < \infty.
\end{align}

{\bf Term $\E X_k B_{kb} X_0^2  - \E X_k B_{kb} \E X_0^2$:} We may argue as before. Since

\begin{align*}
\E X_k B_{kb} X_0^2  - \E X_k B_{kb} \E X_0^2 = \E X_k B_{kb} (X_0^2 - \E X_0^2),
\end{align*}
we can use \eqref{eq:square:star} and the argument for \eqref{prop:variance:expansion:eq:4} to conclude
\begin{align}
\sum_{k \in \N} \big|\E X_k B_{kb} X_0^2  - \E X_k B_{kb} \E X_0^2 \big| < \infty.
\end{align}

{\bf Term $\E X_k B_{kb} X_0 B_{0b} - \E X_k B_{kb} \E X_0 B_{0b}$:}
We first consider the case where $k \leq 2b$. By the Cauchy-Schwarz inequality and Lemma \ref{lem:sipwu} in the Appendix, we have
\begin{align*}
\big\|X_k \sum_{h = 1}^{k-1} X_{k-h} - X_k^{\ast} \sum_{h = 1}^{k-1} X_{k-h}^{\ast} \big\|_2 &\leq \big\|X_k-X_k^{\ast}\big\|_4 \big\|\sum_{h = 1}^{k-1} X_{k-h} \big\|_4 \\&+ \big\|X_k^{\ast}\big\|_4 \big\|\sum_{h = 1}^{k-1} (X_{k-h} - X_{k-h}^{\ast}) \big\|_4 \\&\lesssim \lambda_{k4}\sqrt{k} + 1.
\end{align*}
Due to H\"{o}lder's inequality and Lemma \ref{lem:sipwu} in the Appendix, we thus obtain the bound
\begin{align*}
\Big| \E X_k \sum_{h = 1}^{k-1} X_{k-h} X_0 B_{0b} - \E X_k \sum_{h = 1}^{k-1} X_{k-h} \E X_0 B_{0b} \Big| &\lesssim \big(\lambda_{k4}\sqrt{k} + 1\big) \big\|X_0\big\|_4 \big\|B_{0b}\big\|_4 \\&\lesssim \big(\lambda_{k4}\sqrt{k} + 1\big) \sqrt{b}.
\end{align*}
Moreover, since $\E X_k^{\ast} \sum_{h = k}^{b} X_{k-h} X_0 B_{0b}= 0$, H\"{o}lder's inequality implies
\begin{align*}
\big|\E X_k \sum_{h = k}^{b} X_{k-h} X_0 B_{0b} \big| &\leq \big\|X_k - X_k^{\ast}\big\|_4 \big\|\sum_{h = k}^{b} X_{k-h}\big\|_4 \big\|X_0\big\|_4 \big\| B_{0b} \big\|_4 \\&\lesssim \lambda_{k4}b,
\end{align*}
where we also used Lemma \ref{lem:sipwu} in the Appendix. Since clearly $\big|\E X_k B_{kb} \E X_0 B_{0b}\big| \lesssim 1$ (cf. Lemma \ref{lem:bound:cov}), it follows that
\begin{align}\label{prop:variance:expansion:eq:6}
\sum_{k \leq 2b} \big|\E X_k B_{kb} X_0 B_{0b} - \E X_k B_{kb} \E X_0 B_{0b}\big| \lesssim b \sqrt{b}.
\end{align}
For $k \geq 2b$ we have, arguing similarly as above, that
\begin{align}\label{prop:variance:expansion:eq:7} \nonumber
&\sum_{k \geq 2b}\big|\E X_k B_{kb} X_0 B_{0b} - \E X_k B_{kb} \E X_0 B_{0b}\big| \\
& \lesssim \sum_{k \geq 2b} \sqrt{b} \lambda_{k4} + \sum_{k \geq 2b} \sum_{h = 1}^{b} \lambda_{k-h,4} \lesssim 1.
\end{align}
Combining both \eqref{prop:variance:expansion:eq:6} and \eqref{prop:variance:expansion:eq:7}, we get
\begin{align}\label{prop:variance:expansion:eq:8}
\sum_{k \in \N}\big|\E X_k B_{kb} X_0 B_{0b} - \E X_k B_{kb} \E X_0 B_{0b}\big| \lesssim b \sqrt{b}.
\end{align}

Piecing all bounds together, we finally arrive at
\begin{align*}
&\sum_{k \in \Z}\Big|\E X_k X_0 - \E (Y_k-\E Y_k) (Y_0 - \E Y_0) \Big| \\
&\lesssim (1+x^2)\big(n^{-1/2} + b \sqrt{b} n^{-1}\big) \lesssim (1+x^2)n^{-1/2}
\end{align*}
due to $b \lesssim n^{1/3}$. 
\end{proof}

\subsection{Wasserstein distance}\label{sec:proof:wasser}

We employ a similar smoothing argument as in the proof of Theorem 6.1 in \cite{jirak:tams:2021}. The studentization leads to some additional complications though.

To this end, for $a > 0$ and $b \in \N$ even, let $H_{ab}$  be a real valued random variable with density function
\begin{align}\label{eq_g_ab}
h_{ab}(x) = \co_{ab} a \Big|\frac{\sin(a x)}{a x} \Big|^b, \quad x \in \R,
\end{align}
for some constant $\co_{ab} > 0$. It is well-known (cf. ~\cite{Bhattacharya_rao_1976_reprint_2010}, Section 10) that for even $b$ the Fourier transform $\hat{h}_{ab}$  \hypertarget{ghatab:eq35}{satisfies}
\begin{align}\label{eq_thm_smooth_fourier}
\hat{h}_{ab}(t) =  \left\{
\begin{array}{ll}
2 \pi \co_{ab} u^{\ast \, b}[-a,a](t) &\text{if $|t| \leq a b$},\\
0 &\text{otherwise},
\end{array}
\right.
\end{align}
where $u^{\ast \, b}[-a,a]$ denotes the $b$-fold convolution of \hypertarget{XXKdm:eq35}{the} density of the uniform distribution on $[-a,a]$, that is, $u[-a,a](t) = \frac{1}{2a} \ind_{[-a,a]}(t)$. For $b \geq 6$, let $(H_k)_{k \in \Z}$ be i.i.d. with $H_k \stackrel{d}{=} H_{ab}$, independent of $S_n$ and $\hat{\sigma}_{nb}^{\tau_n}$. Define
\begin{align}\label{defn_diamond_mod}\nonumber
{X}_k^{\diamond} &= X_k + H_{k} - H_{k-1}, \quad {S}_n^{\diamond} = \sum_{k = 1}^n {X}_k^{\diamond} = S_n + H_n - H_0,\\
Y_k^{\diamond}(x) &= X_k^{\diamond} - \frac{x}{2\sqrt{n}{\sigma}_{nb}}X_k^2 - \frac{x}{\sqrt{n}{\sigma}_{nb}}X_k B_{kb}, \quad S_n^{\diamond}(x) = \sum_{k = 1}^n Y_k^{\diamond}(x).
\end{align}
Note that since $b \geq 6$, exploiting also the independence of $(H_k)_{k \in \Z}$ and $S_n$, we have by \eqref{eq_g_ab} and \eqref{eq_thm_smooth_fourier}
\begin{align}\nonumber
&\E H_k = 0, \quad \E|H_k^{}|^4 < \infty,\\ \label{eq_char_S_n_diamond_zero}
&\Big|\E e^{\ic \xi S_n^{\diamond}/\sqrt{n}} \Big| = \Big|\E e^{\ic \xi S_n^{\diamond}(x)/\sqrt{n}} \Big| = 0
\end{align}
for $|\xi| > \sqrt{n} |ab|$ and any $x \in \R$. Denote with
\begin{align}\label{defn_moments}
\big(\sigma_n^{\diamond}(x)\big)^2 = n^{-1}\E \big(S_n^{\diamond}(x)\big)^2, \quad \big(\kappa_n^{\diamond}(x)\big)^3 = n^{-\frac{3}{2}}\E \big(S_n^{\diamond}(x)\big)^3,
\end{align}
 \hypertarget{Phi:eq5}{and} the \hypertarget{littlephi:eq5}{formal} second-order Edgeworth expansion \hypertarget{psin:eq5}{as}
\begin{align}\label{defn_PSI_m}
\Psi_n^{\diamond}\bigl(x,y\bigr) = \Phi\bigl(x\bigr) + \frac{1}{6}\Big({\kappa}_n^{\diamond}(y)/\sigma_n^{\diamond}(y)\Big)^3\bigl(1 - x^2 \bigr) \phi\bigl(x\bigr), \quad x,y \in \R,
\end{align}
where $\Phi$ is the distribution function of a standard normal random variable and $\phi$ its density. Note that due to Lemma \ref{lem:bound:Y:diff} and Lemma \ref{lem:third:cumulant} in the Appendix, we have (recall $\tau_n \asymp \sqrt{\log n}$)
\begin{align}\label{eq:thm:w:cumulant:diamond}
\sqrt{n} \sup_{|x| \leq \tau_n} \big|\big({\kappa}_n^{\diamond}(x)\big)^3 \big| \leq C
\end{align}
for some constant $C$ depending only on $\sigma^2$ and $\Theta_{\mathfrak{a}p}$. Arguing as in \eqref{eq:var:decomposition}, exploiting also the properties of $H_k^{\diamond}$, it follows that
\begin{align}\label{eq:thm:w:sigma:diamond}
n \big(\sigma_n^{\diamond}(x)\big)^2 = n \sigma^2 + O\big(n^{2-\ad} + (1+x^2)\sqrt{n}\big),
\end{align}
uniformly for $|x| \leq \tau_n$. For the proof, we require the following result, which is an immediate consequence of Theorem 2.7 in ~\cite{jirak:tams:2021} (use the comment below \eqref{eq:theta:lambda:estimate:intro} to verify the conditions) and \eqref{eq_char_S_n_diamond_zero}.

\begin{lem}\label{lem:edge:classic}
Let $X_{1},\ldots,X_n$ be a strictly stationary sequence such that $X_k \in \mathcal{E}_k$.
Suppose there exist absolute constants $c^{\dag},C^{\dag} > 0$ such that for $3 < q < 4$
\begin{itemize}
  \item[(i)] $\|X_k\|_q \leq C^{\dag}$, $\E X_k = 0$,
  \item[(ii)] $\sum_{1 \leq k \leq n} k^{\ad^{\dag}} \sup_{l \geq k}\|X_l-X_l'\|_q \leq C^{\dag}$ for $\ad^{\dag} > 5/2$,
  \item[(iii)] $\sum_{|k| \leq n} \E X_0 X_k \geq c^{\dag}$.
\end{itemize}
Then there exists a constant $C^{\ddag}$, only depending on $c^{\dag}$, $C^{\dag}$ and $\delta$, such that
\begin{align*}
\sup_{x \in \R}\Big|\P\big(S_n^{\diamond} \leq \|S_n^{\diamond}\|_2 x \big) - \Psi_n^{\diamond}\big(x,0\big)\Big| \leq C^{\ddag} n^{1-q/2 + \delta},
\end{align*}
where $\delta > 0$ can be selected arbitrarily small.
\end{lem}

\begin{proof}[Proof of Theorem \ref{thm:mein:W}]
Recall that $\hat{\sigma}^{\tau_n}_{nb} = \hat{\sigma}_{nb} \vee \tau_n^{-1}$. Due to \eqref{eq:bound:A:events} and $p > 6$, we have
\begin{align}\label{eq:thm:w:sets}
\P\big(\mathcal{A}_1^c\big) + \P\big(\mathcal{A}_2^c\big) \lesssim n^{-1/2 - \delta}, \quad \delta > 0.
\end{align}
It follows that
\begin{align*}
\E (\hat{\sigma}_{nb}^{\tau_n})^{-1} &= \int_0^{\tau_n} \P\big((\hat{\sigma}_{nb}^{\tau_n})^{-1} > x \big) dx = \int_0^{\tau_n} \P\big(1 > x^2 \hat{\sigma}_{nb}^2 \big) dx
\\&\leq \tau_n \P\big(|\hat{\sigma}_{nb}^2 - \sigma^2| \geq \sigma^2/2 \big) +  \int_0^{\tau_n} \ind_{\{2 > x^2 \sigma^2 \}} dx \\&\lesssim
\tau_n \P\big(\mathcal{A}_1^c\big) + \tau_n \P\big(\mathcal{A}_2^c\big) + 1 \lesssim 1.
\end{align*}

By the triangle inequality, exploiting also the independence of $(H_k)_{k \in \Z}$ and $\hat{\sigma}_{nb}^{\tau_n}$, we conclude from the above
\begin{align*}
W_1\Big(\P_{\frac{S_n}{\sqrt{n} \hat{\sigma}_{nb}^{\tau_n}}}, \P_{G\frac{\sigma_{b}}{\sigma}}\Big) \leq W_1\Big(\P_{\frac{S_n^{\diamond}}{\sqrt{n} \hat{\sigma}_{nb}^{\tau_n}}}, \P_{G\frac{\sigma_{b}}{\sigma}}\Big) +  O\Big(\frac{1}{\sqrt{n}}\Big).
\end{align*}

It is well-known that we can rewrite the Wasserstein distance as
\begin{align}\label{wasserstein_rep_3}
W_1\Big(\P_{\frac{S_n^{\diamond}}{\sqrt{n} \hat{\sigma}_{nb}^{\tau_n}}}, \P_{G\frac{\sigma_{b}}{\sigma}}\Big) = \int_{\R}\Big|\P\Big(\frac{S_n^{\diamond}}{\sqrt{n} \hat{\sigma}_{nb}^{\tau_n}} \leq x \Big) - \Phi\Big(\frac{x \sigma_{b}}{\sigma}\Big)\Big| d x.
\end{align}
We split up this integral into the three regions
\begin{align*}
\mathrm{I}_1 = \{|x| \leq \tau_n\},\, \mathrm{I}_2 = \{\tau_n < |x| \leq \tau_n^2\} \,\, \text{and $\mathrm{I}_3 = \{|x| > \tau_n^2\}$,}
\end{align*}
and show that for each region the corresponding integral is of magnitude $O(n^{-1/2})$.\\
{\bf Case $\mathrm{I}_3$:} Using Lemma \ref{lem:fn} in the Appendix, we obtain for $x \geq \tau_n^2$
\begin{align*}
\P\Big(\frac{|S_n^{\diamond}|}{\sqrt{n} \hat{\sigma}_{nb}^{\tau_n}} > x \Big) \leq \P\Big(\frac{|S_n^{\diamond}|}{\sqrt{n}} > \frac{x}{\tau_n} \Big) \lesssim x^{-2} n^{-1/2},
\end{align*}
and hence, since $\P(X \leq x) = 1 - \P(X > x)$, we get, using standard Gaussian tail bounds, that
\begin{align}
\int_{\mathrm{I}_3} \Big|\P\Big(\frac{S_n^{\diamond}}{\sqrt{n} \hat{\sigma}_{nb}^{\tau_n}} \leq x \Big) - \Phi\Big(\frac{x \sigma_{b}}{\sigma}\Big)\Big| d x \lesssim \frac{1}{\sqrt{n}}.
\end{align}
{\bf Case $\mathrm{I}_2$:} Employing the bound in \eqref{eq:truncation:bound} (recall $p > 6$), we conclude that for $x \geq \tau_n$
\begin{align*}
\P\Big(\frac{|S_n^{\diamond}|}{\sqrt{n} \hat{\sigma}_{nb}^{\tau_n}} > x \Big) \lesssim n^{-1/2} \tau_n^{-2}.
\end{align*}
Hence, using again Gaussian tail bounds, we have
\begin{align}
\int_{\mathrm{I}_2} \Big|\P\Big(\frac{S_n^{\diamond}}{\sqrt{n} \hat{\sigma}_{nb}^{\tau_n}} \leq x \Big) - \Phi\Big(\frac{x \sigma_{b}}{\sigma}\Big)\Big| d x \lesssim \frac{1}{\sqrt{n}}.
\end{align}
{\bf Case $\mathrm{I}_1$:} We first note that
\begin{align*}
\sup_{x \in \R}\Big|\P\Big(\frac{S_n^{\diamond}}{\sqrt{n} \hat{\sigma}_{nb}^{\tau_n}} \leq x \big) - \P\Big(\frac{S_n^{\diamond}}{\sqrt{n} \hat{\sigma}_{nb}} \leq x  \Big)\Big| \leq \P\Big(\hat{\sigma}_{nb} < \tau_n^{-1} \Big).
\end{align*}
For $n$ large enough, we get from $\sigma^2 > 0$ and \eqref{eq:thm:w:sets} (recall $\tau_n \asymp \sqrt{\log n}$)
\begin{align*}
\P\big(\hat{\sigma}_{nb} < \tau_n^{-1} \big) \leq \P\big(\mathcal{A}_1^c\big) + \P\big(\mathcal{A}_2^c\big) \lesssim n^{-1/2 - \delta}, \quad  \delta > 0.
\end{align*}
Arguing as in the proof of Theorem \ref{thm:mein:be}, we conclude from \eqref{eq:thm:w:sets} and the above that uniformly for $|x| \leq \tau_n$
\begin{align*}
&\P\Big(\frac{S_n^{\diamond}(x)}{(\sigma_b^2 n)^{1/2}}\leq x - \mu \Big) - O\big(n^{-1/2 - \delta}\big) \leq \P\Big(\frac{S_n^{\diamond}}{\sqrt{n} \hat{\sigma}_{nb}^{\tau_n}} \leq x  \Big) \\&\leq \P\Big(\frac{S_n^{\diamond}(x)}{(\sigma_b^2 n)^{1/2}}\leq x + \mu \Big) + O\big(n^{-1/2 - \delta}\big), \quad \delta > 0,
\end{align*}
where $\mu \lesssim \frac{|x| b^2 \sigma_b \log n}{n}$ and $S_n^{\diamond}(x)$ is given in \eqref{defn_diamond_mod}. An application of Lemma \ref{lem:edge:classic} (with $q = p/2 > 3$,  $\ad > 4$ in conjunction with Lemma \ref{lem:bound:Y:diff} implies the validity of condition (ii), and \eqref{eq:thm:w:sigma:diamond} validates (iii)) gives
\begin{align*}
\sup_{|x| \leq \tau_n}\Big|\P\Big(\frac{S_n^{\diamond}(x)}{(\sigma_b^2 n)^{1/2}}\leq x \pm \mu \Big) - \Psi_n^{\diamond}\Big(\frac{(x \pm \mu)\sigma_b}{\sigma_n^{\diamond}(x)},x\Big) \Big| \lesssim n^{-1/2 - \delta}, \quad \delta > 0,
\end{align*}
where $\Psi_n^{\diamond}$ is defined in \eqref{defn_PSI_m}. Using \eqref{eq:thm:w:cumulant:diamond}, \eqref{eq:thm:w:sigma:diamond} and Taylor expansion, we get
\begin{align*}
\int_{\mathrm{I}_1} \Big|\Psi_n^{\diamond}\Big(\frac{(x \pm \mu)\sigma_b}{\sigma_n^{\diamond}(x)},x\Big) - \Phi\Big(\frac{x \sigma_{b}}{\sigma}\Big)\Big| d x \lesssim \frac{1}{\sqrt{n}} + \frac{b^2 \log n}{n} \lesssim \frac{1}{\sqrt{n}},
\end{align*}
where we also used $|\mu| \lesssim \frac{|x| b^2 \sigma_b \log n}{n}$ and \hyperref[A4]{\Afour} in the last inequality\footnote{Note that the term $\frac{b^2 \log n}{n}$ in the  first inequality is by a factor $\sqrt{\log n}$ better than what we require, allowing to slightly weaken the conditions on $b$ in \hyperref[A4]{\Afour}.}. Together with the above, this yields
\begin{align}
\int_{\mathrm{I}_1} \Big|\P\Big(\frac{S_n^{\diamond}}{\sqrt{n} \hat{\sigma}_{nb}^{\tau_n}} \leq x \Big) - \Phi\Big(\frac{x \sigma_{b}}{\sigma}\Big)\Big| d x \lesssim \frac{1}{\sqrt{n}}.
\end{align}
Combining all bounds, the claim follows.
\end{proof}

\subsection{Examples}\label{sec:example}

\begin{proof}[Proof of Proposition \ref{prop:sde}]
As can be readily seen from the proof below, we can assume without loss of generality $\delta = 1$ and $l = t/\delta$. Let $(Y_t(x))_{t \geq 0}$ be the diffusion started at $Y_0 = x$. Using Assumption \ref{ass:sde}, an application of It\^{o}'s Formula, Gronwall's inequality and a stopping argument yields
\begin{align*}
\E \|Y_t(x) - Y_t(x')\|_{\R^d}^2 \leq \|x-x'\|_{\R^d}^2 \exp(- 2 t \gamma),
\end{align*}
see Equation (4.6) in ~\cite{Djellout:AoP:2004} for more details. Let $(Y_t')_{t \geq 0}$ be an independent copy of $(Y_t)_{t \geq 0}$. Then due to representation \eqref{rep:sde} and stationarity, we conclude (for $t \in \N$) that $Y_t^* = Y_t(Y_0')$ and
\begin{align}
\E \|Y_{t}^{*} - Y_{t}\|_{\R^d}^2 \leq 4 \E\|Y_0\|_{\R^d}^2 \exp(- 2 t \gamma).
\end{align}
Next, note that for any $K > 0$, we have
\begin{align*}
\E \|Y_t - Y_t^*\|_{\R^d}^p \leq K^{p-2} \E \|Y_t - Y_t^*\|_{\R^d}^2 + \E \|Y_t - Y_t^*\|_{\R^d}^p \1_{\|Y_t - Y_t^*\|_{\R^d} > K}.
\end{align*}
Moreover, for any $Y \geq 0$ and $q > p \geq 1$, the inequality
\begin{align*}
\E Y^p \1_{Y > K} \leq K^{p-2} \E Y^2 + \frac{p}{q-p}K^{p-q}\E Y^q
\end{align*}
holds. Hence, by the above and $\sup_{t} \E \|Y_t\|_{\R^d}^q < \infty$, there exists $c > 0$ such that for $p = 7$ we have
\begin{align*}
\E \|Y_t - Y_t^*\|_{\R^d}^p \lesssim \exp(-c t).
\end{align*}
This immediately implies $\E \|Y_t - Y_t'\|_{\R^d}^p \lesssim  \exp(-c t)$.
\end{proof}

\section*{Acknowledgements}

I am indebted to the reviewers for a careful reading of the manuscript. The comments and suggestions have been very beneficial, significantly increasing the quality of the paper.

\appendix

\section{}\label{sec:appendix}

To increase readability, we list some inequalities here we use multiple times. Throughout this section, we assume that $\|X_k\|_p < \infty$ for appropriate $p \geq 2$ (see the results below), $X_k \in \mathcal{E}_k$, $\mathcal{E}_k =\sigma(\varepsilon_j, \, j \leq k)$ with $(\varepsilon_j)_{j \in \Z}$ i.i.d. taking values in some measurable space, and that $X_k$ is strictly stationary (the results actually hold in more generality). Finally, let $S_n = X_1 + \ldots + X_n$.

The first result corresponds to a straightforward modification of Theorem 2 (ii) in \cite{Wu_fuk_nagaev}.

\begin{lem}\label{lem:fn}
Let $|a_i| \leq 1$, $a_i \in \R$, and suppose that for $p > 2$, we have
\begin{align*}
\sum_{j \geq m} \|X_j - X_j'\|_p \leq C m^{-\alpha}, \quad \alpha > 1/2 - 1/p.
\end{align*}
Then there exists a constant $c > 0$, depending on $C$, $p$, $\alpha$ and $\Theta_{0p} = \|X_0\|_p + \sum_{j = 1}^{\infty} \|X_j - X_j'\|_p$, such that for all $x \geq 1$
\begin{align*}
&\P\Big(\max_{k \leq n}\Big|\sum_{i = 1}^k a_i X_i \Big| \geq c x \Big) \leq \frac{n}{x^p} + \exp\Big(-\frac{x^2}{n}\Big).
\end{align*}
\end{lem}

The next result is Theorem 1 (iii) in \cite{wu_2005}.

\begin{lem}\label{lem:theta:lambda:relation}
For each $p \geq 2$, there exists a constant $C > 0$, such that
\begin{align*}
\big\|X_k - X_k^{\ast}\big\|_p^2 \leq C \sum_{l \geq k}\big\|X_l - X_l'\big\|_p^2.
\end{align*}
\end{lem}

The next lemma follows from Lemma 9.1 (i) in \cite{jirak:tams:2021} together with Lemma \ref{lem:theta:lambda:relation} above. 

\begin{lem}\label{lem:third:cumulant}
Suppose that $\sum_{k \in \N} k^{\ad} \|X_k - X_k'\|_3 < \infty$, $\ad > 7/2$. Then there exists a constant $C > 0$, such that
\begin{align*}
\big|\E S_n^3 \big| \leq C n.
\end{align*}
\end{lem}

Finally, we restate parts of Theorem 3 in \cite{wu_2011_asymptotic_theory}.

\begin{lem}\label{lem:sipwu}
Let $p \geq 2$, and suppose that $\sum_{k \in \N} \|X_k - X_k'\|_p < \infty$. Then there exists a constant $C > 0$, such that
\begin{align*}
\big\|S_n\big\|_p \leq C \sqrt{n}.
\end{align*}
\end{lem}

\end{document}